\documentclass{article}
\usepackage[left=2.5cm,top=2cm,right=2.5cm,bottom=2cm]{geometry}

\usepackage{tikz}
	\usetikzlibrary{math}
\usepackage{amsthm, amssymb, amsmath, amsrefs, cancel}
\usepackage[capitalise, noabbrev]{cleveref}
\usepackage{xcolor}
\usepackage{appendix}
\usepackage{multirow}
\usepackage{array}
\usepackage{bold-extra}

\theoremstyle{plain}
	\newtheorem{theorem}{Theorem}[section]

	\newtheorem{lemma}[theorem]{Lemma}

\theoremstyle{definition}
	\newtheorem{example}[theorem]{Example}
	\newtheorem{definition}[theorem]{Definition}
\theoremstyle{remark}

\DeclareMathOperator{\STS}{STS}
\newcommand{\Z}{\mathbb{Z}}
\newcommand{\nofil}{\textsc{Nofil}}

\newcommand{\nodekayles}[0]{\textsc{Node Kayles}}

\usepackage{xspace}
\makeatletter
\DeclareRobustCommand\onedot{\futurelet\@let@token\@onedot}
\def\@onedot{\ifx\@let@token.\else.\null\fi\xspace}

\makeatother
\usepackage{soul}

\begin{document}
\begin{center}
\uppercase{\bf Minimal Graph Embeddings via Point Deletions in Steiner triple systems}
\vskip 20pt
{\bf Melissa A. Huggan}\footnote{Supported by the Natural Sciences and Engineering Research Council of Canada (funding reference number RGPIN-2023-03395 and DGECR-2023-00190)
.}\\
{\small Department of Mathematics, Vancouver Island University\\Nanaimo, British Columbia, Canada}\\
{\tt melissa.huggan@viu.ca}\\ 
\vskip 10pt
{\bf Svenja Huntemann}\footnote{Supported by the Natural Sciences and Engineering Research Council of Canada (funding reference number RGPIN-2022-04273 and DGECR-2022-00452).}\\
{\small Department of Mathematics and Statistics, Mount Saint Vincent University\\Halifax, Nova Scotia, Canada}\\
{\tt svenja.huntemann@msvu.ca}\\ 
\vskip 10pt
{\bf Brett Stevens}\footnote{Supported by the Natural Sciences and Engineering Research Council of Canada (funding reference number RGPIN-2023-04668).}\\
{\small School of Mathematics and Statistics, Carleton University\\Ottawa, Ontario, Canada}\\
{\tt brett@math.carleton.ca}\\ 
\end{center}
\begin{abstract}
The game \nofil{} is a two-player combinatorial game in which players take turns marking points of a design such that the set of marked points does not contain a block. Equivalently, we can think of the points as being deleted from the design and points that are on singleton sets can no longer be marked. Every game play eventually results in the design becoming a graph. Previous work has shown that every graph is reachable from some Steiner triple system ($\STS$), although the order of the constructed $\STS$ is often far from the known lower bounds. In this paper we give embeddings of complete graphs and star graphs into a $\STS$ that is minimal or very nearly meets the bounds. We further discuss possible minimal embeddings of empty graphs, paths, and cycles.
\end{abstract}

\section{Introduction}

In a previous paper \cite{HugganHS2021}, we defined a game, \nofil{}, played on a Steiner triple system and noted its relation to an existing game, \nodekayles{}, played on graphs. We proved that for any graph and all sufficiently large $v \equiv 1,3 \pmod 6$, there exists a Steiner triple system on $v$ points in which the graph is embedded in such a way that playing \nofil{} on the Steiner triple system can reduce to playing \nodekayles{} on the chosen graph.  From this we proved a PSPACE completeness result for \nofil{}.  The Steiner triple systems for those embeddings are very large.  In this paper we are interested in understanding what are the smallest possible Steiner triple systems into which a given graph can be embedded. First, we review the objects of our study and previous results in more detail.

A \emph{combinatorial game} is a two player game of perfect information and no chance, where players move alternately.  A well-studied winning convention is \emph{normal play}, where the player without a move on their turn loses. The interested reader can learn more about combinatorial game theory from \cites{ANW, Siegel}.  \nodekayles{} is an impartial combinatorial game that is played on a graph.

\begin{definition}[\nodekayles{} \cite{schaefer_78}] 
The \emph{board} is a graph.  On their turn, a player chooses an unplayed vertex from the graph as long as it is not adjacent to any other previously chosen vertices. Chosen vertices are simply considered \emph{played}, regardless of the player who chose them (this is {\em impartiality}). Any vertex adjacent to a played vertex is unplayable, i.e.\ it is forbidden to create an edge with all points played. If every unplayed point is also unplayable, then no move is possible, the game ends, and the last player to move wins. 
\end{definition}

In a previous paper we defined a game which generalizes \nodekayles{} to hypergraphs \cite{HugganHS2021}.  
\begin{definition}[\nofil{}: {\bf N}ext {\bf O}ne to {\bf F}ill {\bf I}s the {\bf L}oser]
The \emph{board} is a hypergraph. On their turn, a player chooses an unplayed vertex from the hypergraph as long as it does not complete a hyperedge of {\em played} points. Chosen points are simply considered \emph{played}, regardless of the player who chose them. If all but one point from any hyperedge have been played, the last point is unplayable. If every unplayed point is also unplayable, the game ends, and the last player to move wins. 
\end{definition}

The set of played vertices at the end of the game forms a maximal independent set in the sense that it does not contain any hyperedge \cite{MR309807}. We note that the game can be played on any hypergraph but we are interested in playing the game on Steiner triple systems. 
 
\begin{definition}[\cite{Handbook}]
A \emph{Steiner triple system} ($\STS(v)$) ($V$, $\mathcal{B}$) is a set $V$ of $v$ elements together with a collection $\mathcal{B}$ of $3$-subsets (blocks or triples) of $V$ with the property that every $2$-subset of $V$ occurs in exactly one block $B \in \mathcal{B}$. The size of $V$ is the \emph{order} of the $\STS$.
\end{definition}
We will call the vertices and hyperedges of a $\STS$ \emph{points} and \emph{blocks}, respectively. For the game played on a Steiner triple system, if two out of three points of a block have been played, the last point is unplayable. 

As play proceeds in a $\STS$, or indeed in any hypergraph, and more points have been played and more points become unplayable, eventually the combinatorial structure which represents the remaining playable points can be described by a graph, and the game becomes equivalent to \nodekayles{} on that graph.  In \cite{HugganHS2021} we proved that for any graph $G$ and any $v \equiv 1,3 \pmod 6$ sufficiently large, there exists a $\STS(v)$ and a line of play of \nofil{} that results in the exact equivalent game as \nodekayles{} played on $G$.  Using this embedding theorem and the fact that determining the winner of \nodekayles{} on graphs is PSPACE-complete \cite{schaefer_78} we showed that determining the winner of a partially played game of \nofil{} on Steiner triple systems is PSPACE-complete under randomized reductions \cite{HugganHS2021}.   

Our embedding theorem produces Steiner triple systems that are, in general, much larger than the embedded graph \cite{HugganHS2021}.  In this paper, we investigate the feasibility of minimally embedding some graphs in Steiner triple systems.  In \cref{sec_background} we give some background on \nofil{} and how play eventually becomes equivalent to playing \nodekayles{} on a graph.  We also describe the combinatorics of partitioning the points of the $\STS$ into Played, Available, and Unplayable points.  Cardinality restrictions between these sets are also reviewed.  In \cref{skolem_sequences} we present the definition of Skolem sequences which we will need in our constructions. In \cref{sec_constructions} we construct minimal embeddings of complete graphs and embeddings of stars that are either minimal or attain the second or third smallest admissible order.  In \cref{sec_other} we describe the results of experiments on paths, cycles, and empty graphs and discuss the challenges of finding minimal embeddings for these graph families.

\section{Background}\label{sec_background}

At any point of game play in \nofil{} the points can be partitioned into three groups:
\begin{itemize}
\item \textbf{Played}: Those that have been chosen already. We denote this set by $P$.  The rules of the game imply that for every block $B$ in the design, $B \not\subseteq P$.
\item \textbf{Unplayable}: If all points but one in a block have been played, the last point is unplayable.  We denote the set of such points by $U$. 
\item \textbf{Available}: The set of unplayed points that are still playable is denoted $A$.
\end{itemize}
In \cref{sec_game_play} we will describe how these  point sets evolve as play proceeds. In \cref{sec_restrictions} we describe constraints of the relative cardinality of these sets.

\subsection{Process of play in \nofil{}} \label{sec_game_play}
On each turn, the set of played points, $P$, increases by one. The remaining points can be partitioned between those that are unplayable, $U$, because they appear on a triple with two already played points, and those that are available, $A$, which have at most one played point in every block on which they appear. If we concentrate on just the playable points, we can restrict our attention to the induced hypergraph on $A$, inherited from the $\STS$ and called the {\em available hypergraph}, which encodes the available plays and any restrictions among points. 
At any point in the game a hyperedge either contains unplayable points or not. If a hyperedge has unplayable points, then it imposes no restriction on its other points. This is because with an unplayable point, it could never become filled with played points. 
Thus, this hyperedge does not appear in the available hypergraph.  In the second case, a hyperedge without any unplayable points still constrains the play on the remaining available points of the hyperedge because it is forbidden to fill this hyperedge with played points.  This hyperedge, with the played points removed, is a hyperedge of the available hypergraph. 

In the research literature on \nodekayles{} a version of this available hypergraph is used to simplify play \cite{fleischer_06}:  deleting any played vertex and all its neighbours is the equivalent of the process described above.  In \cref{ex: example} we give an example of game play, these  point sets, and the available hypergraph on a $\STS(9)$. 

\begin{example}\label{ex: example}
Consider a $\STS(9)$ with $V = \{1, 2, 3, 4, 5, 6, 7, 8, 9\}$ and blocks
\[\mathcal{B}=\{123, 456, 789, 147, 258,  369, 159, 267, 348, 168, 249, 357\}.\] 

A game on this $\STS(9)$ could play out as summarized in \cref{ex: steiner 9}, where the columns $P$, $A$, and $U$ give the respective disjoint sets of points described above. 
The \emph{blocks} column shows the list of blocks of the design, where points within a block are overlined if they have been played. No block is allowed to have all three of its points overlined. The \emph{available hypergraph} column describes the restrictions of points within blocks and at Turn 0 matches the $\STS(9)$ itself. In this column, ``$23$'', for example, indicates that points $2$ and $3$ are on a block together, and both cannot be chosen because the third point in the block has been played. Consider a block $348$ and suppose $3$ becomes unplayable. Then on the next turn, the $48$ hyperedge is removed, because this block imposes no further restriction on play of either $4$ or $8$, as the block can never be completed. The last column, \emph{hypergraph figure} presents the modification to the $\STS(9)$ throughout the example game play. The greyed out vertices represent unplayable points. The blacked out vertices represents the played points. All games played on the unique $\STS(9)$ are equivalent to this sequence of play, because the automorphism group is 3-transitive. After turn 4, the game is over and Player 2 has won the game. 
\renewcommand{\arraystretch}{1.5}
\begin{table}[ht!]
\begin{center}
\begin{tabular}{|c|c|c|c|c|c|c|}\hline
\multirow{2}{*}{Turn}&\multirow{2}{*}{$P$}&\multirow{2}{*}{$A$}&\multirow{2}{*}{$U$}&\multirow{2}{*}{Blocks}&Available&Hypergraph\\
&&&&& Hypergraph&Figure\\\hline\hline
 \multirow{4}{*}{$0$}&&\multirow{4}{*}{\begin{minipage}[c]{1.5cm}
        \begin{tabular}{c}
            {$1,2,3,$} \\
            {$4,5,6,$} \\
            {$7,8,9$} \\
        \end{tabular}
    \end{minipage}}&&$123, 456, 789,$ &$123, 456, 789,$&\multirow{4}{*}{\includegraphics[width=25mm, height=25mm]{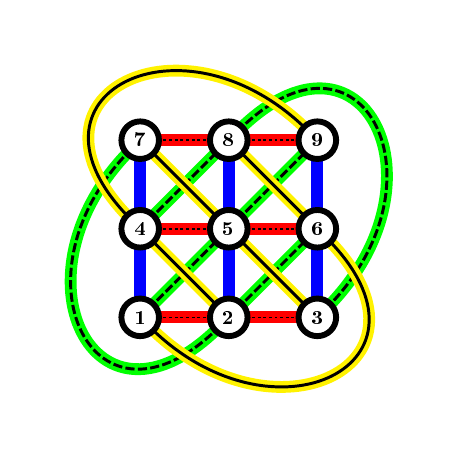}}
\\
&&&&$ 147, 258, 357,$&$ 147, 258, 357,$&\\
    &&&& $369, 159, 267,$&$369, 159, 267,$&\\%\hline
    &&&& $348, 168, 249$&$ 348, 168, 249$&\\\hline%  $369, 59, 267, 348, 68, 249$\\\hline
 \multirow{4}{*}{$1$}&\multirow{4}{*}{$1$}&&&$\overline{1}23, 456, 789,$ &$23, 456, 789,$&\multirow{4}{*}{\includegraphics[width=25mm, height=25mm]{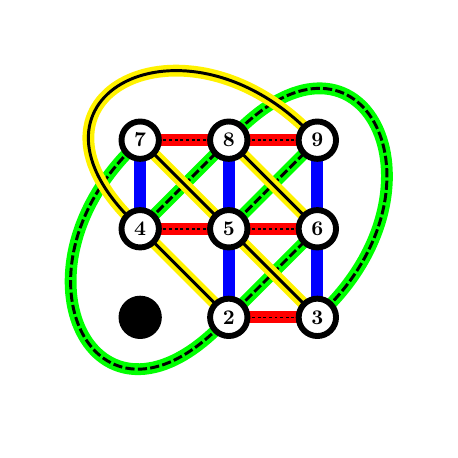}}
\\
&&$2, 3, 4, 5, $&&$ \overline{1}47, 258, 357,$&$ 47, 258, 357,$&\\
    &&$6, 7, 8, 9$&& $369, \overline{1}59, 267,$&$369, 59, 267,$&\\%\hline
    &&&& $348, \overline{1}68, 249$&$ 348, 68, 249$&\\\hline%  $369, 59, 267, 348, 68, 249$\\\hline
 \multirow{4}{*}{$2$}&\multirow{4}{*}{$1, 2$}&&\multirow{4}{*}{$3$}&$\overline{1}\overline{2}3, 456, 789, $ &$456, 789,$&\multirow{4}{*}{\includegraphics[width=25mm, height=25mm]{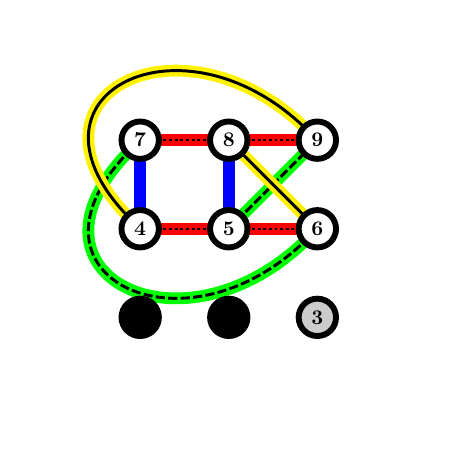}}\\
&&$4, 5, 6,$&&$\overline{1}47, \overline{2}58, 357,$& $ 47, 58,$&\\
    && 7, 8, 9&& $369, \overline{1}59, \overline{2}67,$&$59, 67,$& \\
&&&&$ 348, \overline{1}68, \overline{2}49$ &$ 68, 49$&\\\hline  
\multirow{4}{*}{$3$}&\multirow{4}{*}{$1, 2, 6$}&\multirow{4}{*}{$4, 5, 9$}&\multirow{4}{*}{$3, 7, 8$}&$\overline{1}\overline{2}3, 45\overline{6}, 789, $&\multirow{4}{*}{ $45, 59, 49$}&\multirow{4}{*}{\includegraphics[width=25mm, height=25mm]{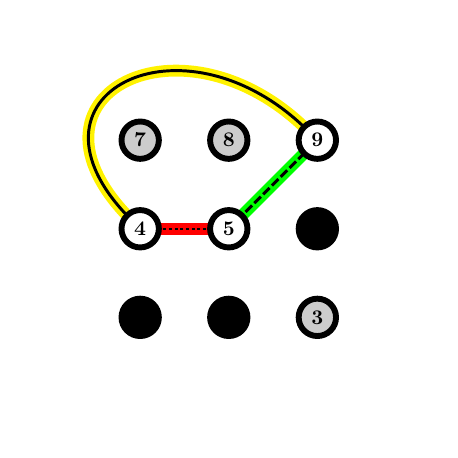}}\\
&&&&$\overline{1}47, \overline{2}58, 357,$ &&\\
    &&&& $3\overline{6}9, \overline{1}59, \overline{2}\overline{6}7,$&&\\
&&&&$ 348, \overline{1}\overline{6}8, \overline{2}49$&&\\\hline  
\multirow{4}{*}{$4$}&\multirow{4}{*}{$1, 2, 6, 4$}&&\multirow{4}{*}{$3, 7, 8, 5, 9$}&$\overline{1}\overline{2}3, \overline{4}5\overline{6}, 789,$ &&\multirow{4}{*}{\includegraphics[width=25mm, height=25mm]{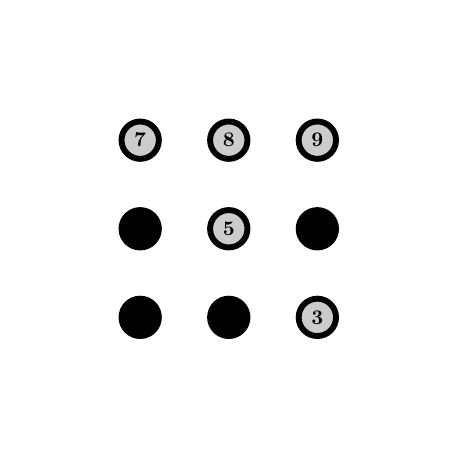}} \\
&&&&$ \overline{1}\overline{4}7, \overline{2}58, 357,$&&\\
&&&& $3\overline{6}9, \overline{1}59, \overline{2}\overline{6}7,$&& \\
&&&&$ 3\overline{4}8, \overline{1}\overline{6}8, \overline{2}\overline{4}9$&&\\\hline
\end{tabular}
\caption{Example of game play on a $\STS(9)$.}\label{ex: steiner 9}
\end{center}
\end{table}
\renewcommand{\arraystretch}{1}
\end{example}

\nofil{} ends when $A = \emptyset$ and the available hypergraph is empty.  Since an empty hypergraph is also a graph, there is a point in every game play when the available hypergraph on $A$ is simply a graph, possibly one with no edges or no vertices. In the above example, this first happens at turn $3$ where the available hypergraph is the graph $K_{3}$ with vertices labeled as $4$, $5$, and $9$. From this point forward the players are effectively playing \nodekayles{} on $K_3$.  Different lines of play in a Steiner triple system may leave different graphs on which \nodekayles{} is played.  For each available hypergraph which is a graph that can be reached in the play of \nofil{} on a $\STS$, we say this graphs is {\em embedded} in the $\STS$.  In this paper we investigate embedding graphs $K_a$, $K_{1,a-1}$, $\overline{K_a}$, $P_a$, and $C_a$ and are interested in finding the smallest Steiner triple systems into which they can be embedded.

\subsection{Restrictions}\label{sec_restrictions}

If $G$ is a graph on  point set $A$, then the following constraints apply when embedding it into a $\STS$:
\begin{itemize}
    \item every point in $U$ must be in some block with two points from $P$;
    \item for each edge $xy \in E(G)$, there must be a block $xyp$ with $p \in P$;
    \item for each non-edge $xy \not \in E(G)$, there must be a block $xyu$ with $u \in U$.
\end{itemize}
If a $\STS$ has a partition of its points into $P$, $A$, and $U$, then satisfying these three constraints is equivalent to $G$ being embedded in the $\STS$.  In our constructions we classify each block of the $\STS$ according to which of $P$, $A$, and $U$ its three points belong.  Recall that every pair of points is on a unique block of the Steiner triple system.  Let
\begin{align*}
  PPU &= \{ B: |B\cap P|=2, |B\cap U| = 1\}, \\
  PAA &= \{ B: |B\cap P|=1, |B\cap A| = 2\}, 
\end{align*}
and define $PAU$, $PUU$, $AAU$, $AUU$, and $UUU$ analogously. The rules of the game forbid any blocks of form $PPP$ or $PPA$.  Since we are embedding graphs whose edges have size two, there will be no blocks of the form $AAA$. Considerations of these facts and letting $p = |P|$, $a = |A|$, and $u = |U|$ we derived the following bounds on $u$ which can be used to determine a lower bound on $v=p+a+u$ of an embedding into a $\STS(v)$ \cite{HugganHS2021}.

\begin{lemma}[\cite{HugganHS2021}] \label{u_bounds}
Let $(X,\mathcal{B})$ be a $\STS(v)$ with points $P$ played.  Suppose the available hypergraph is the graph $G$.  Then
\begin{align}
u &\leq v-a-\chi'(G),    \label{chi_g} \\
  u &\geq \chi'(\overline{G}), \label{chi_comp_g} \\
  u &\geq \frac{v-a-1}{2}, \label{chi_K_p} \\
u &\geq \frac{a(v-1) -4e}{2a},  \label{auu} \\
u &\geq \frac{3v-2a-1 - \sqrt{(v-2a+1)^2 + 16e}}{4}, \label{puu_1} \\
u &\leq \frac{3v-2a-1 + \sqrt{(v-2a+1)^2 + 16e}}{4},    \label{puu_2} \\
  u &\leq v-a - \frac{2e}{a},  \label{pau} \\
  u &\leq v-a + \frac{1-\sqrt{8(v-a)+1}}{2},  \label{u_binom_p_2} 
\end{align}
and if $v^2-4v \leq 24e$, then either
\begin{equation}\label{uuu_1}
u \geq \frac{v}{2} + \frac{\sqrt{72e-3v^2+12v}}{6}
\end{equation}
or
\begin{equation}\label{uuu_2}
u \leq \frac{v}{2} - \frac{\sqrt{72e-3v^2+12v}}{6}.
\end{equation}

\end{lemma}
We will use these bounds in addition to the fact that $p$, $a$, and $u$ must be non-negative integers to establish how close to minimal the $v$ are in our constructions of embeddings into $\STS(v)$.

\subsection{Skolem sequences}\label{skolem_sequences}

Several of our constructions will take advantage of the existence of Skolem sequences and variants thereof.
\begin{definition}
  A {\em Skolem sequence of order $t$} is a collection of $t$ pairs $(a_r,b_r)$, where $1 \leq r \leq t$, such that $b_r-a_r =r$ for all $1 \leq r \leq t$ and $\bigcup_{1 \leq r \leq t} \{a_r,b_r\} = \{1,2,\ldots,2t\}$.   A {\em hooked Skolem sequence of order $t$} is a collection of $t$ pairs $(a_r,b_r)$, where $1 \leq r \leq t$, such that $b_r-a_r =r$ for all $1 \leq r \leq t$ and $\bigcup_{1 \leq r \leq t} \{a_r,b_r\} = \{1,2,\ldots,2t-1,2t+1\}$.  A {\em split Skolem sequence of order $t$} is a collection of  $t$ pairs $(a_r,b_r)$, $1 \leq r \leq t$ such that $b_r-a_r =r$ for all $1 \leq r \leq t$ and $\bigcup_{1 \leq r \leq t} \{a_r,b_r\} = \{1,2,\ldots,t,t+2,\ldots,2t+1\}$.   A {\em Langford sequence of order $t$ and defect $d$} is a collection of $t$ pairs $(a_r,b_r)$ where $d \leq r < t+d$, such that $b_r-a_r = r$ and $\bigcup_{1 \leq r \leq t} \{a_r,b_r\} = \{1,2,\ldots,2t\}$.  A {\em hooked Langford sequence of order $t$ and defect $d$} is a collection of $t$ pairs $(a_r,b_r)$ where $d \leq r < t+d$, such that $b_r-a_r = r$ and $\bigcup_{1 \leq r \leq t} \{a_r,b_r\} = \{1,2,\ldots,2t-1,2t+1\}$. 
\end{definition}

\begin{example}
The collection of pairs
\[
  \{(1,2),(5,7),(3,6),(4,8)\}
  \]
forms a Skolem sequence of order 4. An example of a hooked Skolem sequence of order $6$ is 
  \begin{equation} \label{hook_6}
      %\{    (10,11),(2,4),(6,9),(1,5),(3,8),(7,13)\}.
\{ (9,10),(1,3),(4,7),(2,6),(8,13),(5,11)\}
  \end{equation}
The pairs
  \[
    \{(1,2),(7,9),(3,6),(4,8) \}
    \]
form a split Skolem sequence of order $4$.  And the pairs 
\[
\{(4,6),(8,11),(1,5),(2,7),(3,9)\}
\]
form a hooked Langford sequence of order 5 and defect 2.
\end{example}

Necessary and sufficient conditions for the existence of Skolem sequences, hooked Skolem sequences, split Skolem sequences and Langford sequences were proven in 1957 by Skolem, in 1961 by O'Keefe, in 1966 by Rosa, and in 1983 by Simpson, respectively.

\begin{theorem}
\begin{enumerate}
\item \cite{MR0092797} There exists a Skolem sequence of order $t$ if and only if $t \equiv 0,1 \pmod 4$.
\item \cite{MR125024} There exists a hooked Skolem sequence of order $t$ if and only if $t \equiv 2,3 \pmod 4$.
\item \cite{MR211884} There exists a split Skolem sequence of order $t$ if and only if $t \equiv 0,3 \pmod 4$.
\item \cite{simpson} There exist a Langford sequence of order $t$ and defect $d$ if and only if $t \geq 2d-1$ and 
\[
t \equiv \begin{cases} 0,1 \pmod 4 & \mbox{ if }d \equiv 1 \pmod 2 \\ 0,3 \pmod 4 & \mbox{ if }d \equiv 0 \pmod 2.\end{cases}
\]
\item \cite{simpson} There exist a hooked Langford sequence of order $t$ and defect $d$ if and only if $t(t - 2d + 1) \geq -2$ and 
\[
t \equiv \begin{cases} 2,3 \pmod 4 & \mbox{ if }d \equiv 1 \pmod 2 \\ 1,2 \pmod 4 & \mbox{ if }d \equiv 0 \pmod 2.\end{cases}
\]
\end{enumerate}
\end{theorem}

In two cases we will need Skolem sequences with specific properties.

\begin{lemma}\label{special_skolem}
    For all $1 \leq t \equiv 0,1 \pmod 4$ there exists a Skolem sequence of order $t$ with $b_1 = a_1 + 1 = 2$.
\end{lemma}
\begin{proof}
    If $t=1$, then the Skolem sequence is trivial: $\{(1,2)\}$.  Otherwise for $4 \leq t$ the desired sequence is constructed by adjoining the pair $(1,2)$ to a Langford sequence of order $t-1$ and defect 2 that has been shifted up by $2$. 
\end{proof}

\begin{lemma}\label{special_hooked}
    For all $6 \leq t \equiv 2,3 \pmod 4$ there exists a hooked Skolem sequence of order $t$ with $b_2 = a_2 + 2= 3$.
\end{lemma}
\begin{proof}
For $t = 6$ the hooked Skolem sequence is given in \cref{hook_6}.  For $t=7$ the sequence is 
\[
\{ (5, 6), (1, 3), (10, 13), (8, 12), (2, 7), (9, 15), (4, 11) \},
\]
for $t=10$ it is 
\[
     \{(4, 5), (1, 3), (9, 12), (11, 15), (2, 7), (13, 19), (14, 21), (10, 18), (8, 17), (6, 16)\},
\]
and for $t=11$ it is
\[
     \{(4, 5), (1, 3), (6, 9), (17, 21), (14, 19), (10, 16), (8, 15), (12, 20), (2, 11), (13, 23), (7, 18)\}.
\]
For $13 \leq t \equiv 2,3 \pmod 4$ these can be made by taking the union of the Skolem sequence $\{(6,7),(1,3),(2,5),(4,8)\}$ and a hooked Langford sequence of order $t-4$ and defect $5$ that has been shifted to the up by 8.
\end{proof}

\section{Embedding graphs}\label{sec_constructions}

\cref{sec_restrictions} gave restrictions on the sizes of $P$, $A$, and $U$ and required structures of some of the blocks in any embedding of a given graph in a $\STS$. In this section we consider embeddings in a $\STS(v)$ where $v$ is at or close to the bounds from \cref{sec_restrictions}.  To test the bounds we are interested in families of graphs which are diverse in terms of edge densities, maximum degree, and edge chromatic numbers since these are relevant in the bounds.  In this section we successfully embed  $K_a$ minimally and $K_{1,a-1}$ minimally or near minimally.  Later in \cref{sec_other} we will describe explorations of empty graphs, paths, and cycles.

\subsection{Embedding Complete Graphs}\label{complete_graph}

In this section, we will give minimal embeddings for complete graphs.

\begin{theorem} \label{k_a_theorem}
The complete graph $K_a$ can be embedded in a $\STS(3a)$ if $a$ is odd and a $\STS(3a+1)$ if $a$ is even. In both cases these are the minimum orders possible to embed $K_a$.  
\end{theorem}
\begin{proof}

First consider the case when $a$ is odd. Since $\chi'(K_a) = a$, \cref{chi_g} gives that $p \geq a$. \cref{chi_K_p} gives that $u \geq p-1 \geq a-1$, but no $\STS(3a-1)$ exists because $3a-1$ is even.  So an embedding of $K_a$ into a $\STS(3a)$ is minimal when $a$ is odd.

Let $V =  \mathbb{Z}_{a} \times \mathbb{Z}_3$, and define the blocksets
  \begin{align*}
    T_k &=   \{ \{(i,k), (j,k), ((i+j)/2,k-1)\} :\; i\neq j \in \mathbb{Z}_a\} \text{ for }k\in\mathbb{Z}_3,\text{ and } \\
    L &= \{ \{(i,0), (i,1), (i,2)\} :\; i \in \mathbb{Z}_a\}.
  \end{align*}
The blocks $L \cup \bigcup_{k \in \mathbb{Z}_3} T_k$ form a $\STS(3a)$ with point set $V$. Let 
\begin{align*}
    P &= \mathbb{Z}_{a} \times \{0\}, \\
    A &= \mathbb{Z}_{a} \times \{1\}, \\
    U &= \mathbb{Z}_{a} \times \{2\},
    \end{align*} partition $V$. We have that $AAU$, $PUU$, and $UUU$ are empty.  The $\binom{a}{2}$ blocks of $T_1$ are in $PAA$ and ensure that every pair of vertices from $A$ forms an edge of $K_a$.  The $\binom{a}{2}$ blocks of $T_0$ are in $PPU$ and guarantee that every vertex of $U$ is unplayable.  In turn, every pair of points from $P$ is on a block of $PPU$, so all the points of $A$ are playable.  The $\binom{a}{2}$ blocks of $T_2$ are in $AUU$ and the blocks of $L$ are in $PAU$.

Secondly, when $a = 2\ell$ is even, then $\chi'(K_a) = a-1$ and \cref{chi_g} implies that $p \geq a-1$.  \cref{chi_K_p} again gives that $u \geq p-1 \geq a-2$ and thus $v = p+a+u \geq 3a-3 \equiv 3 \pmod 6$. The graph $K_a$ has $e = a(a-1)/2$ edges and if $v = 3a-3$, we compute that $v^2 - 4v \leq 24e$ and thus one of \cref{uuu_1} or \cref{uuu_2} must hold.  However, neither holds when $u = a-2$.  Since $3a-3 \equiv 3 \pmod 6$, the next highest order for which a $\STS$ exists is $3a+1$.  Thus an embedding of $K_a$ into a $\STS(3a+1)$ is minimal when $a$ is even.

Let
\begin{align*}
P &= \{p_{i,j} :\; i \in \mathbb{Z}_{\ell}, j \in \{0,1\}\},\\
A &= \{a_{i,j} :\; i \in \mathbb{Z}_{\ell}, j \in \{0,1\}\},\text{ and}\\
U &= \{u_{i,j} :\; i \in \mathbb{Z}_{\ell}, j \in \{0,1\}\} \cup \{u_\infty\}.
\end{align*}
To construct an embedding of $K_a$ into a $\STS(3a+1)$ we need three edge decompositions of complete graphs on each of $P$, $A$, and $U$.

We first construct $\{\mathcal{P}_x\}_{x \in \mathbb{Z}_{\ell} \times \{0,1\} \cup \{\infty\}}$, an edge decomposition of the complete graph on $P$ into matchings such that $\mathcal{P}_{\infty} = \{p_{i,0},p_{i,1} :\; i \in \mathbb{Z}_{\ell}\}$ and $p_{i,0}$ and $p_{i,1}$ both have degree zero only in $\mathcal{P}_{i,0}$ and $\mathcal{P}_{i,1}$.  To construct this, let $K_{2^{\ell}}$ be the complete multipartite graph with $\ell$ parts of size 2.  Let the $i$th part be $\{(i,0),(i,1)\}$.  Erzurumluo\u{g}lu and Rodger have shown that $K_{2^{\ell}}$ can be decomposed into $\ell$ $(2\ell-2)$-cycles where the $i$th part is missed by exactly one $(2\ell-2)$-cycle, $C_i$ \cite{MR3322805}.  Decomposing each $(2\ell-2)$-cycle, $C_i$ into alternating edges gives the desired matchings, $\mathcal{P}_{i,0}$ and $\mathcal{P}_{i,1}$, each missing precisely $p_{i,0}$ and $p_{i,1}$.  We obtain $\mathcal{P}_{\infty}$ from the parts of the multipartite graph.

Next we construct $\{\mathcal{A}_y\}_{y \in \mathbb{Z}_{\ell} \times \{0,1\}}$, an edge decomposition of the complete graph on $A$ such that $\mathcal{A}_{i,0}$ is a perfect matching for all $i \in \mathbb{Z}_{\ell}$ and $\mathcal{A}_{i,1}$ is a matching with $a/2 -1$ edges in which $a_{i,0}$ and $a_{i,1}$ have degree zero.  This edge decomposition can be constructed in $\mathbb{Z}_{a}$ where
\[
\mathcal{A}_y = \begin{cases} \{\{i-k,i+k+1\} :\; 0 \leq k < a/2 \} & y=(i,0), 0 \leq i < a/2 \\ \{\{i-k,i+k\} :\; 1 \leq k < a/2 \} & y=(i,1), 0 \leq i < a/2, \end{cases}
\]
followed by a relabeling of the vertices.

Finally we construct $\{\mathcal{U}_z\}_{z \in \mathbb{Z}_{\ell} \times \{0,1\} \cup \{\infty\}}$, a near 1-factorization of the complete graph on $U$ such $u_z$ has degree zero in $\mathcal{U}_z$ and $\mathcal{U}_{\infty} = \{ \{u_{i,0}, u_{i,1}\}:\; i \in \mathbb{Z}_{\ell}\}$. This edge decomposition can be constructed on  point set $\mathbb{Z}_{a+1}$ where $\mathcal{U}_i = \{\{i-j,i+j\} :\; 1 \leq j \leq a/2\}$ followed by a relabeling of the points.

Using these partitions into matchings, we construct the blocks
\begin{align*}
  PPU & = \{ e \cup u_{\infty} :\; e \in \mathcal{P}_{\infty} \} 
     % &\quad
 \cup \{e \cup u_{i,j} :\; e \in \mathcal{P}_{i,j}\}, \\
  PAA &= \{ e \cup p_{i,j} :\; e \in \mathcal{A}_{i,j} \}, \\
  PAU &= \{ \{p_{i,1}, a_{i,j}, u_{i,j}\} :\; i \in \mathbb{Z}_{\ell}, j \in \{0,1\}\},\\
  PUU &= \{ \{p_{i,0}, u_{i,0}, u_{i,1} \} :\; i \in \mathbb{Z}_{\ell} \}, \\
  AUU &= \{ a_{i,j} \cup e :\; e \in \mathcal{U}_{i,j}\}, \\
  AAU &= \emptyset, \\
  UUU &= \emptyset.
\end{align*}
These blocks all together form a $\STS(3a+1)$. All the edges within $P$ are contained in the blocks $PPU$ since all edges from $\{\mathcal{P}_x\}_{x \in \mathbb{Z}_{\ell} \times \{0,1\} \cup \{\infty\}}$ are used therein.  Similarly all edges within $A$ are used in the blocks $PAA$.  The edges in $\mathcal{U}_{\infty}$ are used in $PUU$ and the rest of the edges in the $\mathcal{U}_{i,j}$ are all used in $AUU$.  The blocks of $PAA$ contain all edges between $P$ and $A$ except $\{p_{i,1},a_{i,j}\}$ for $i \in \mathbb{Z}_{\ell}$ and $j \in \{0,1\}$ which are covered in blocks of $PAU$.   The blocks of $PPU$ use the edges $\{p_{i,j}, u_{\infty}\}$ and $\{p_{i,j}, u_{i',k}\}$ for all $i\neq i'  \in \mathbb{Z}_{\ell}$ and $j,k \in \{0,1\}$.  For the rest of the edges between $P$ and $U$, $\{p_{i,0},u_{i,k}\}$ appear in blocks of $PUU$ and $\{p_{i,1},u_{i,k}\}$ edges appear in blocks of $PAU$ for all $i \in \mathbb{Z}_{\ell}$. The blocks of $AUU$ contain all the edges between $A$ and $U$ except $\{a_{i,j},u_{i,j}\}$ because $u_{i,j}$ is the only point of zero degree in $\mathcal{U}_{i,j}$.  These missing pairs are covered by blocks of $PAU$. We have

\setlength{\arraycolsep}{2pt}
\[
  \begin{array}{ccccccccccccccc}
  |PPU|& +& |PAA|& + &|PAU|& + &|PUU|& +& |AAU|& +& |AUU|& +& |UUU|& =& \\
  \binom{a}{2} &+& \binom{a}{2} &+& a &+& \frac{a}{2} &+& 0 &+&  \frac{a^2}{2} &+& 0 &=& \frac{a(3a+1)}{2}
  \end{array}
\]
which is the number of blocks in a $\STS(3a+1)$.
\setlength{\arraycolsep}{6pt}

Finally we note that every point of $U$ is on a block from $PPU$ and is thus unplayable, no points of $A$ are on blocks with two played points and so they remain playable and all edges in $A$ are on a block from $PAA$ and are thus in the graph.
\end{proof}

\subsection{Embedding Stars}\label{stars}

In this section we embed stars $K_{1,a-1}$ in Steiner triple systems.  If $a = 2$, then $K_{1,1} = K_2$ and the minimal embedding is constructed in \cref{k_a_theorem}.  The edge chromatic number of $K_{1,a-1}$ is $a-1$ so \cref{chi_g} gives $p \geq a-1$. With $e=a-1$, \cref{auu} gives $u \geq p+a-3$ for $2 \leq a \leq 3$ and $u \geq p+a-4$ for $a \geq 4$.  These combined give that $v \geq 4a-6$.  This lower bound is general and does not depend on $v \equiv 1,3 \pmod 6$.  In our constructions below we will apply this additional constraint as appropriate to each modular class of $a$ that we are considering.

All of our constructions of $\STS$ with embedded stars are modifications of Cho's constructions of 4-rotational Steiner triple systems \cite{MR0677049}. Like Cho, we will present the point set of the STS as a set together with a group action on the set.  This group action defines orbits of the blocks and we present one block from each orbit.  All other blocks can be obtained by applying the group action to the blocks presented.

\begin{lemma}
When $a=4,6$, the stars $K_{1,3}$ and $K_{1,5}$ can be minimally embedded in a $\STS(19)$.  Otherwise, let $10 \leq a \equiv 0,4 \pmod 6$. Then $K_{1,a-1}$ can be embedded in a $\STS(4a-3)$. When $a \equiv 4 \pmod 6$ this embedding is minimal and when $a \equiv 0 \pmod 6$ this embedding is in a $\STS$ of the next smallest possible order.    
\end{lemma}
\begin{proof}
When $a \equiv 0 \pmod 6$ the fact that $v \equiv 1,3 \pmod 6$ gives that $v \geq 4a-5$.  \cref{star_a_4} shows a minimal embedding of $K_{1,5}$ in a $\STS(19)$.  For $12 \leq a$, our construction does not construct an embedding in a smallest possible $\STS$ but it does give the next smallest possible order of a $\STS$: $v=4a-3$.  When $a \equiv 4 \pmod 6$, the fact that $v \equiv 1,3 \pmod 6$ gives that $v \geq 4a-3$.  When $a=4$, we get $e=3$ and embedding $K_{1,3}$ in a $\STS(13)$ or $\STS(15)$ is ruled out by a more in-depth analysis of the bounds \cite{HugganHS2021}.   Therefore the smallest possible triple system is a $\STS(19)$ and such an embedding is given in \cref{star_a_4}.  When $10 \leq a \equiv 4 \pmod 6$ we construct the minimal possible embedding of $K_{1,a-1}$ in a $\STS(4a-3)$.  

\renewcommand{\arraystretch}{1.5}
\begin{table}[ht]
\begin{center}
\scalebox{0.68}{
\begin{tabular}{|c|>{\centering\arraybackslash}p{10cm}|p{0.5cm}|c|>{\centering\arraybackslash}p{10cm}|} \cline{1-2} \cline{4-5}
 $P$ & $\overline{5}$, $\overline{6}$, $\overline{9}$, $\overline{11}$, $\overline{17}$ & & $P$ & $\overline{1}$, $\overline{2}$, $\overline{3}$, $\overline{8}$, $\overline{10}$, $\overline{13}$  \\ \cline{1-2} \cline{4-5}
 $A$ &  $\underline{2}$, $\underline{4}$, $\underline{13}$, $\underline{14}$, $\underline{16}$, $\underline{18}$ &&  $A$ &  $\underline{0}$, $\underline{9}$, $\underline{15}$, $\underline{17}$ \\\cline{1-2} \cline{4-5}
 $U$ & 0, 15, 7, 12, 10, 8, 1, 3 && $U$ & 4, 5, 6, 7, 11, 12, 14, 16, 18 \\\cline{1-2}\cline{4-5}
  \multirow{3}{*}{$PPU$} & \multirow{3}{*}{ \begin{tabular}{c} \{$\overline{5}$, $\overline{9}$, 0\}, \{$\overline{6}$, $\overline{11}$, 0\}, \{$\overline{5}$, $\overline{11}$, 3\}, \{$\overline{5}$, $\overline{6}$, 7\}, \{$\overline{5}$, $\overline{17}$, 1\},\\
   \{$\overline{9}$, $\overline{11}$, 8\},     \{$\overline{6}$, $\overline{9}$, 15\},     \{$\overline{9}$, $\overline{17}$, 7\},     \{$\overline{11}$, $\overline{17}$, 10\},     \{$\overline{6}$, $\overline{17}$, 12\} 
   \end{tabular}} & & \multirow{3}{*}{$PPU$} & 
  \{$\overline{3}$, $\overline{8}$, 11\}, \{$\overline{3}$, $\overline{13}$, 18\},     \{$\overline{3}$, $\overline{10}$, 7\},    \{$\overline{1}$, $\overline{3}$, 6\},    \{$\overline{2}$, $\overline{3}$, 5\},       \\
   && & &  \{$\overline{8}$, $\overline{13}$, 16\},     \{$\overline{2}$, $\overline{10}$, 16\},     \{$\overline{1}$, $\overline{8}$, 18\},     \{$\overline{8}$, $\overline{10}$, 4\},     \{$\overline{2}$, $\overline{8}$, 12\}, \\
  &&&  &  \{$\overline{1}$, $\overline{2}$, 4\},     \{$\overline{1}$, $\overline{10}$, 14\},     \{$\overline{1}$, $\overline{13}$, 12\},     \{$\overline{2}$, $\overline{13}$, 11\},    \{$\overline{10}$, $\overline{13}$, 5\} \\\cline{1-2} \cline{4-5}
  $PAA$ & \{$\overline{17}$, $\underline{13}$, $\underline{14}$\}, \{$\overline{5}$, $\underline{14}$, $\underline{18}$\}, \{$\overline{9}$, $\underline{4}$, $\underline{14}$\}, \{$\overline{11}$, $\underline{2}$, $\underline{14}$\}, \{$\overline{6}$, $\underline{14}$, $\underline{16}$\} && $PAA$ & \{$\overline{3}$, $\underline{15}$, $\underline{17}$\}, \{$\overline{2}$, $\underline{0}$, $\underline{15}$\}, \{$\overline{10}$, $\underline{9}$, $\underline{15}$\} \\\cline{1-2} \cline{4-5} 
  \multirow{4}{*}{$PAU$ }&\{ $\overline{5}$, $\underline{13}$, 10\},     \{$\overline{9}$, $\underline{13}$, 3\},     \{$\overline{11}$, $\underline{13}$, 15\},    \{$\overline{6}$, $\underline{13}$, 1\},    \{$\overline{17}$, $\underline{18}$, 0\}, && \multirow{4}{*}{$PAU$ }&
\{$\overline{3}$, $\underline{9}$, 16\},     \{$\overline{3}$, $\underline{0}$, 4\},    \{$\overline{1}$, $\underline{0}$, 16\},    \{$\overline{8}$, $\underline{0}$, 7\},    \{$\overline{8}$, $\underline{17}$, 14\},      \\
   &     \{$\overline{5}$, $\underline{2}$, 15\},    \{$\overline{5}$, $\underline{4}$, 8\},    \{$\overline{5}$, $\underline{16}$ , 12\},    \{$\overline{9}$, $\underline{2}$, 12\},    \{$\overline{9}$, $\underline{18}$, 1\}, &&&  \{$\overline{8}$, $\underline{9}$, 6\},    \{$\overline{8}$, $\underline{15}$, 5\},    \{$\overline{2}$, $\underline{9}$, 18\},    \{$\overline{13}$, $\underline{9}$, 4\},    \{$\overline{1}$, $\underline{9}$, 7\}, \\
  &     \{$\overline{9}$, $\underline{16}$, 10\},    \{$\overline{11}$, $\underline{18}$, 7\},    \{$\overline{11}$, $\underline{4}$, 12\},    \{$\overline{11}$, $\underline{16}$, 1\},    \{$\overline{17}$, $\underline{4}$, 15\}, &&&  \{$\overline{2}$, $\underline{17}$, 7\},    \{$\overline{13}$, $\underline{15}$, 7\},    \{$\overline{13}$, $\underline{0}$, 14\},    \{$\overline{10}$, $\underline{0}$, 12\},    \{$\overline{1}$, $\underline{17}$, 5\}, \\
&    \{$\overline{6}$, $\underline{4}$, 10\},    \{$\overline{17}$, $\underline{2}$, 8\},    \{$\overline{6}$, $\underline{2}$, 3\},    \{$\overline{6}$, $\underline{18}$, 8\},    \{$\overline{17}$, $\underline{16}$, 3\} &&& \{$\overline{1}$, $\underline{15}$, 11\},    \{$\overline{10}$, $\underline{17}$, 11\},    \{$\overline{13}$, $\underline{17}$, 6\} \\\cline{1-2} \cline{4-5}
  $PUU$ & $\emptyset$ && $PUU$ & \{$\overline{3}$, 12, 14\}, \{$\overline{10}$, 6, 18\}, \{$\overline{2}$, 6, 14\} \\ \cline{1-2} \cline{4-5}
  \multirow{2}{*}{$AAU$} & \{$\underline{4}$, $\underline{13}$, 0\},     \{$\underline{2}$, $\underline{13}$, 7\},    \{$\underline{13}$, $\underline{18}$, 12\},    \{$\underline{13}$, $\underline{16}$, 8\},    \{$\underline{2}$, $\underline{16}$, 0\}, && \multirow{2}{*}{$AAU$} & \multirow{2}{*} {\{$\underline{0}$, $\underline{17}$, 18\}, \{$\underline{9}$, $\underline{17}$, 12\}, \{$\underline{0}$, $\underline{9}$, 5\} }\\
&    \{$\underline{16}$, $\underline{18}$, 15\},     \{$\underline{4}$, $\underline{16}$, 7\},    \{$\underline{2}$, $\underline{4}$, 1\},    \{$\underline{4}$, $\underline{18}$, 3\},    \{$\underline{2}$, $\underline{18}$, 10\} && &\\\cline{1-2} \cline{4-5}
 \multirow{2}{*}{$AUU$} & \multirow{2}{*}{\{$\underline{14}$, 0, 8\}, \{$\underline{14}$, 3, 15\}, \{$\underline{14}$, 7, 10\}, \{$\underline{14}$, 1, 12\}} &&  \multirow{2}{*}{$AUU$} &\{$\underline{15}$, 16, 18\},    \{$\underline{17}$, 4, 16\},    \{$\underline{15}$, 4, 14\},\\
&&&&    \{$\underline{9}$, 11, 14\},     \{$\underline{15}$, 6, 12\},    \{$\underline{0}$, 6, 11\} \\\cline{1-2} \cline{4-5}
  \multirow{2}{*}{$UUU$}&   \{0, 1, 15\},    \{0, 7, 12\},    \{0, 3, 10\},    \{7, 8, 15\},    \{10, 12, 15\}, && \multirow{2}{*}{$UUU$}& \{7, 14, 16\},    \{11, 12, 16\},    \{5, 6, 16\},    \{4, 11, 18\},    \{7, 12, 18\},\\
&    \{1, 3, 7\},    \{3, 8, 12\},    \{1, 8, 10\} &&  &     \{5, 14, 18\},    \{4, 6, 7\},    \{4, 5, 12\},    \{5, 7, 11\} \\\cline{1-2} \cline{4-5}
  \end{tabular}}
  \end{center}
  \caption{$\STS(19)$s with $K_{1,5}$ and  $K_{1,3}$ embedded respectively. \label{star_a_4}}
\end{table}
\renewcommand{\arraystretch}{1}

The  point set of the $\STS(4a-3)$ will be 
  \[
    V = \left (\Z_{a-1} \times \{1,2,3,4\}\right ) \cup \{\infty\}.
  \]
  where the second elements in the Cartesian product, $\{1,2,3,4\}$,  are written as subscripts.  The element $g \in \Z_{a-1}$ acts on this set via $g(i_{j}) = (i+g)_j$ and $g(\infty) = \infty$. We let
  \begin{align*}
    P &= \Z_{a-1} \times \{1\}, \\
    A &= \left (\Z_{a-1} \times \{2\}\right ) \cup \{\infty\}, \\
    U &= \Z_{a-1} \times \{3,4\}.
  \end{align*}
We break the construction into cases of $a \pmod {24}$.  We will present the $\STS$ by giving one block per orbit of the action. 
\begin{itemize}
    \item   When $a \equiv 6 \pmod {24}$, let $a = 6t$. Since $a > 6$, we know $t>1$. Let $(a_r,b_r)$ be a hooked Skolem sequence of order $3t-1$.  Let $(c_r,d_r)$ be a Skolem sequence of order $t-1$. The base blocks of the $\STS(4a-3)$ are given in the top left of \cref{0_6}.
\item When $a \equiv 12 \pmod{24}$, let $a = 6t$. Let $(a_r,b_r)$ be a Skolem sequence of order $3t-1$. Let $(c_r,d_r)$ be a Skolem sequence of order $t-1$. The base blocks are shown in the top right of \cref{0_6}.

\item  When $a \equiv 0,18 \pmod{24}$, let $a= 6t$, so $t \geq 3$.  Let $(a_r,b_r)$ be a split Skolem sequence of order $3t-1$ and let $(c_r,d_r)$ be a hooked Skolem sequence of order $t-1$. The base blocks are given in the bottom of \cref{0_6}.

\renewcommand{\arraystretch}{1.5}
\begin{table}[ht]
  \begin{center}
    \scalebox{0.71}{
\begin{tabular}{cc}
\begin{tabular}{|c|>{\centering\arraybackslash}p{10cm}|}\hline
  \multirow{2}{*}{$PPU$} &   $\{0_1,r_1,(b_r)_4\},\; 1  \leq r \leq 3t-1,\; r \neq 4$ \\
  &     $\{0_1,4_1,2_3\}$ \\\hline
{$PAA$}&     $\{\infty,0_2,(-3)_1\},$  \\\hline
  \multirow{3}{*}{$PAU$}& $ \{(2r)_1,0_2,r_3\},\; 1 \leq r \leq 6t-4,\; r \neq 2,3t-2,3t-1$ \\
                         & $ \{0_1, 0_2,(-1)_4\}, \{(-2)_1,0_2,(3t-1)_3\}, \{(-1)_1,0_2,(-1)_3\}, $\\
                         & $ \{4_1,0_2,(b_4)_4\}, \{0_1,4_2,(b_4)_4\}$ \\\hline
$PUU$ & $ \{(-1)_1,0_3,(3t-1)_3\}$ \\\hline
  \multirow{2}{*}{$AAU$} &  $ \{0_2,r_2,(b_r)_4\},\; 1 \leq r \leq 3t-1,\; r \neq 4$ \\
  &  $\{0_2,4_2,2_3\}$  \\\hline
  {$AUU$} &    $ \{\infty,0_3,1_4\}, \{0_2,0_3,(3t-2)_3\}$ \\\hline
  \multirow{2}{*}{$UUU$}&    $ \{0_4,r_4,(b_r)_3\},\; 1 \leq r \leq 3t-1$ \\
  & $ \{0_3,r_3,(d_r+t-1)_3\},\; 1 \leq r \leq t-1 $  \\\hline
\end{tabular}
&
\begin{tabular}{|c|>{\centering\arraybackslash}p{10cm}|}\hline
  \multirow{2}{*}{$PPU$} &  $  \{0_1,r_1,(b_r)_4\},\; 1 \leq r \leq 3t-1,\; r \neq 4 $ \\
  &    $ \{0_1, 4_1, 2_3\}$ \\\hline
{$PAA$}& $ \{\infty,0_1, 3_2\}$  \\\hline
  \multirow{3}{*}{$PAU$}& $ \{0_1,(-2r)_2,(-r)_3\},\; 1 \leq r \leq 6t-4,\; r \neq 2,3t-2,3t-1$ \\
                         & $\{0_1,0_2,0_4\},     \{0_1,2_2,(3t+1)_3\},    \{0_1,1_2,0_3\}, $ \\
                         &  $\{4_1,0_2, (b_{4})_4\},    \{0_1,4_2, (b_{4})_4\} $\\\hline
$PUU$ &    $ \{(-1)_1,0_3,(3t-1)_3\}$ \\\hline
  \multirow{2}{*}{$AAU$} &   $  \{0_2,r_2,(b_r)_4\},\; 1 \leq r \leq 3t-1,\; r \neq 4 $  \\
  & $ \{0_2, 4_2, 2_3\}$ \\ \hline
  {$AUU$} &   $  \{\infty,0_3,0_4\},    \{0_2,0_3,(3t-2)_3\} $  \\\hline
  \multirow{2}{*}{$UUU$}&  $\{(b_r)_3, 0_4,r_4\},\; 1 \leq r \leq 3t-1  $ \\
  & $ \{0_3,r_3,(d_r+t-1)_3\},\; 1 \leq r \leq t-1 $ \\\hline
\end{tabular}  \\
& \\
\multicolumn{2}{c}{
\vspace{0.5cm}
\begin{tabular}{|c|>{\centering\arraybackslash}p{10cm}|}\hline
  \multirow{2}{*}{$PPU$} &   $  \{0_1,r_1,(b_r)_4\},\; 1 \leq r \leq 3t-1,\; r \neq 7,$  \\
  &    $ \{0_1, 7_1, 5_3\}, $  \\\hline
{$PAA$ }&    $  \{\infty,0_1,5_2\}$ \\\hline
  \multirow{3}{*}{$PAU$}&    $  \{0_1,(1-2r)_2,(2-r)_3\},\; 1 \leq r \leq 6t-5,\; r \neq 4,3t-1,3t, $\\
                         & $ \{0_1,0_2,(3t)_4\},     \{0_1,2_2,2_3\},     \{0_1,3_2,4_3\}, $  \\
                         & $ \{0_1,1_2,(3t+1)_3\},     \{7_1, 0_2, (b_{7})_4\}, \{0_1, 7_2, (b_{7})_4\} $ \\\hline
$PUU$ &    $ \{0_1,3_3,(3t+2)_3\},  $ \\\hline
  \multirow{2}{*}{$AAU$} &   $  \{0_2,r_2,(b_r)_4\},\; 1 \leq r \leq 3t-1,\; r \neq 7,$  \\
  &    $ \{0_2, 7_2, 5_3\},$  \\ \hline
  {$AUU$} &   $  \{\infty,0_3,(3t-1)_4\}, \{0_2,(3t+1)_3,(-1)_3\},$  \\\hline
  \multirow{2}{*}{$UUU$}&    $ \{0_3,(-b_r)_4,(-a_r)_4\},\; 1 \leq r \leq 3t-1,$ \\
  &    $ \{0_3,r_3,(d_r+t-1)_3\},\; 1 \leq r \leq t-1, $ \\\hline
\end{tabular}}
\end{tabular}}
\end{center}
\caption{The blocks of $\STS(4a-3)$ when $a \equiv 6 \pmod{24}$ (top left), $a \equiv 12\pmod{24}$ (top right), and $a \equiv 0,18 \pmod {24}$ (bottom), respectively. \label{0_6}}
\end{table}
\renewcommand{\arraystretch}{1}

\item   When $a \equiv 4,10 \pmod{24}$ let $a = 6t+4$.  When $a=10$, the minimal embedding of $K_{1,9}$ in a $\STS(37)$ is given in the embedding appendix file.  When $a \geq 28$, let $(a_r,b_r)$ be a Skolem sequence or order $3t+1$.  Let the base blocks of a cyclic $\STS(a-1)$ be $\{c_r,d_r,e_r\}$, where $1 \leq r \leq t+1$ \cite{Handbook}. The base blocks of a $\STS(4a-3)$ are presented in the left of \cref{4_6}. 
\item When $a \equiv 16,22 \pmod{24}$,  let $a = 6t+4$. Let $(a_r,b_r)$ be a hooked Skolem sequence of order $3t+1$.  Let the base blocks of a cyclic $\STS(a-1)$ be $\{c_r,d_r,e_r\}$, where $1 \leq r \leq t+1$ \cite{Handbook}. Let $1 \leq i \leq (3t+1)/2$ and $j=2i$, $k = a_{i} + b_{j} + 3 + a_{2t+1}$ and $l = k-i = a_{j} + b_{i} + 3 + a_{2t+1}$.  Since $(3t+1)/2 > 3$, there is a choice of $i$ such that neither $k$ nor $l$ are zero. The right of \cref{4_6} shows the base blocks of the $\STS(4a-3)$.

\renewcommand{\arraystretch}{1.5}
\begin{table}[ht]
\begin{center}
  \scalebox{0.74}{
\begin{tabular}{cc}
\begin{tabular}{|c|>{\centering\arraybackslash}p{.55\textwidth}|}\hline
  \multirow{2}{*}{$PPU$} &  $   \{0_1,r_1,(b_r)_4\} ,\; 1 \leq r \leq 3t+1,\; r \neq 2t+1 ,$ \\
  &   $  \{(2t+1)_1,0_1,(a_{2t+1}-(2t+1))_3\},$ \\\hline
{$PAA$}&    $ \{\infty, 0_1, (-a_{2t+1})_2\}, $ \\\hline
  \multirow{2}{*}{$PAU$}&    $ \{(a_{2t+1}-r)_1,0_2,r_3\} ,\; 1 \leq r \leq 6t+2,\;r \neq a_{2t+1} \pm (2t+1),$ \\
  &   $  \{0_1,(2t+1)_2,(b_{2t+1})_4\},     \{0_1,(-2t-1)_2,(a_{2t+1})_4\},$ \\\hline
$PUU$ &    $ \{0_1,(-a_{2t+1})_3,0_4\} , $  \\\hline
  \multirow{2}{*}{$AAU$} &   $  \{0_2,r_2,(b_r)_4\} ,\; 1 \leq r \leq 3t+1,\; r \neq 2t+1 , $ \\
  &   $  \{0_2,(-2t-1)_2,(a_{2t+1}+2t+1)_3\}$ \\ \hline
  {$AUU$} &  $   \{0_2,0_3,0_4\} ,     \{\infty, 0_3, (b_{2t+1})_4\}, $   \\\hline
  \multirow{2}{*}{$UUU$}&   $  \{0_3,r_3,(b_r)_4\} ,\; 1 \leq r \leq 3t+1,\; r \neq 2t+1 ,$\\
                         &  $   \{(c_r)_4,(d_r)_4,(e_r)_4\} ,\; 1 \leq r \leq t+1, $\\
  &   $  \{0_3,(2t+1)_3,(4t+2)_3\} , $\\\hline
\end{tabular}
&
\begin{tabular}{|c|>{\centering\arraybackslash}p{.55\textwidth}|}\hline
  \multirow{2}{*}{$PPU$} &  $   \{0_1,r_1,(b_r)_4\} ,\; 1 \leq r \leq 3t+1,\; r \neq j ,$\\
  & $\{0_1,j_1,(2k-1-a_{2t+1})_3 \}, $\\\hline
{$PAA$ }& $  \{\infty, 0_1, (-a_{2t+1}-3)_2\} $  \\\hline
  \multirow{2}{*}{$PAU$}&    $ \{(a_{2t+1}+3-r)_1,0_2,(2+r)_3\} ,\; 1 \leq r \leq 6t+2,\; r \neq k,l, $\\
  & $\{0_1,(b_j+a_i)_2,(b_j)_4, \{0_1,(a_j+b_i)_2,(a_j)_4\}, $ \\\hline
$PUU$ &  $   \{0_1,(-a_{2t+1}-1)_3,(-1)_4\} , $  \\\hline
  \multirow{2}{*}{$AAU$} &   $  \{0_2,(-r)_2,(-b_r)_4\} ,\; 1 \leq r \leq 3t+1,\; r \neq i ,$\\
  & $\{0_2,i_2,(2+k)_3\}, $   \\ \hline
  {$AUU$} &  $   \{\infty, 0_3, (b_{2t+1})_4\} ,     \{0_2,2_3,1_4\} ,   $    \\\hline
  \multirow{2}{*}{$UUU$}&   $  \{(c_r)_4,(d_r)_4,(e_r)_4\} ,\; 1 \leq r \leq t+1, $\\
  &   $  \{0_3,(r)_3,(b_r)_4\} ,\; 1 \leq r \leq 3t+1, r \neq 2t+1 ,$\\
  & $\{0_3,(2t+1)_3,(4t+2)_3\} , $ \\\hline
\end{tabular}
\end{tabular}
}
\end{center}
\caption{The blocks of $\STS(4a-3)$ when $a \equiv 4,10 \pmod{24}$ (left) and $a \equiv 16,22 \pmod{24}$ (right), respectively. \label{4_6}}
\end{table}
\renewcommand{\arraystretch}{1}

\end{itemize}

In each of the cases above, the copy of $K_{1,a-1}$ will be on points $A$ with edges $\{\{\infty, i_{1}\}:\; i \in \Z_{a-1}\}$.   These only appear in blocks of the form $\{\infty, x_1, y_2\}$, which are $PAA$ blocks.  The non-edges of $K_{1,a-1}$ are all of the form $\{x_2,y_2\}$ and these only appear in blocks of the form $\{x_2,y_2,z_4\}$ and $\{x_2,y_2,z_3\}$ which are $AAU$ blocks.  The edges of the complete graph $K_{p}$ on $P$ are all of the form $\{x_1,y_1\}$ and these appear only on blocks of the form $\{x_1,y_1,z_4\}$ and $\{x_1,y_1,z_3\}$ which are $PPU$ blocks.  Finally because there is both a base block of form $\{x_1,y_1,z_3\}$ and one of form $\{x_1,y_1,z_4\}$, every point in $U$ appears on at least one block from $PPU$.  These are the necessary and sufficient conditions for $K_{1,a-1}$ to be embedded in a $\STS$. \end{proof}

\begin{lemma}
When $a=3$, the star $K_{1,2}$ can be minimally embedded in a $\STS(13)$. When $a=7$, the star $K_{1,6}$ can be minimally embedded in a $\STS(25)$.  When $a=8$, the star $K_{1,7}$ can be minimally embedded in a $\STS(27)$.  Otherwise let $9 \leq a \equiv 1,2,3 \pmod 6$ and set $a = 6t+1,6t+2,6t+3$ respectively.  Then $K_{1,a-1}$ can be embedded in a $\STS(24t+7)$.  When $a \equiv 3 \pmod 6$ this is a minimal embedding in a $\STS(4a-5)$.  When $a \equiv 2 \pmod 3$ this is an embedding in a $\STS(4a-1)$, not the minimal order but the next smallest admissible order.   When $a \equiv 1 \pmod 3$ this is an embedding in a $\STS(4a+3)$, not the minimal order but the third smallest possible order.  
\end{lemma}
\begin{proof}
When $a \equiv 1 \pmod 6$ the fact that $v \equiv 1,3 \pmod 6$ gives that $v \geq 4a-3$.  We have verified that a minimal embedding of $K_{1,6}$ in a $\STS(25)$ exists.  Our construction does not construct an embedding in a smallest possible $\STS$ but it does give the third smallest possible order of a $\STS$: $v=4a+3$.    When $a \equiv 2 \pmod 6$, the fact that $v \equiv 1,3 \pmod 6$ shows that the minimal possible $\STS$ would have order $v = 4a-5$.  The star $K_{1,7}$ can be minimally embedded in a $\STS(27)$.  For $14 \leq a$, we construct an embedding in the second smallest admissible order of a $\STS$: $v=4a-1$.  When $a \equiv 3 \pmod 6$, the fact that $v \equiv 1,3 \pmod 6$ shows that the minimal possible $\STS$ would have order $v = 4a-5$.  The bounds determine that for $a=3$ it is not possible to embed $K_{1,2}$ in either a $\STS(7)$ nor $\STS(9)$. \cref{star_a_3} shows a minimal embedding of $K_{1,2}$ in a $\STS(13)$.  We construct the minimal embedding in a $\STS(4a-5)$.  The minimal embeddings of $K_{1,6}$ in a $\STS(25)$ and $K_{1,7}$ in a $\STS(27)$ can be found in the embedding appendix file.

\renewcommand{\arraystretch}{1.5}
\begin{table}[h]
\begin{center}
\begin{tabular}{|c|c|}\hline
 $P$ & $\overline{4}$, $\overline{5}$, $\overline{6}$, $\overline{7}$ \\\hline
 $A$ &  $\underline{0}$, $\underline{3}$, $\underline{9}$ \\\hline
 $U$ & 1, 2, 8, 10, 11, 12\\\hline
\multirow{2}{*}{$PPU$} & 
  \{$\overline{4}$, $\overline{5}$, 10\}, \{$\overline{4}$, $\overline{6}$, 11\},     \{$\overline{4}$, $\overline{7}$, 12\},   \\ 
  & \{$\overline{5}$, $\overline{6}$, 8\},  \{$\overline{5}$, $\overline{7}$, 2\},   \{$\overline{6}$, $\overline{7}$, 1\} \\\hline
 $PAA$ & \{$\overline{5}$, $\underline{0}$, $\underline{9}$\}, \{$\overline{7}$, $\underline{3}$, $\underline{9}$\} \\\hline
  \multirow{2}{*}{$PAU$ }&
\{$\overline{4}$, $\underline{9}$, 1\},     \{$\overline{4}$, $\underline{0}$, 2\},    \{$\overline{4}$, $\underline{3}$, 8\},    \{$\overline{6}$, $\underline{9}$, 2\},    \{$\overline{6}$, $\underline{0}$, 10\}, \\
&     \{$\overline{6}$, $\underline{3}$, 12\},    \{$\overline{5}$, $\underline{3}$, 11\},    \{$\overline{7}$, $\underline{0}$, 8\} \\ \hline
  $PUU$ & \{$\overline{7}$, 10, 11\}, \{$\overline{5}$, 1, 12\} \\\hline
  $AAU$ & \{$\underline{0}$, $\underline{3}$, 1\} \\\hline
  \multirow{2}{*}{$AUU$} &\{$\underline{9}$, 10, 12\},    \{$\underline{9}$, 8, 11\},  \\
  & \{$\underline{3}$, 2, 10\},  \{$\underline{0}$, 11, 12\} \\\hline
  {$UUU$}&\{1, 8, 10\},    \{1, 2, 11\},    \{2, 8, 12\} \\ \hline
\end{tabular}
  \caption{$\STS(13)$ with a $K_{1,2}$ embedded. \label{star_a_3}}
\end{center}
\end{table}
\renewcommand{\arraystretch}{1}

When $a \equiv 1,2,3 \pmod 6$, let $a = 6t+1,6t+2,6t+3$ respectively.  In each case we embed the star in an $\STS(24t+7)$ with point set  
  \[
    V = \left (\Z_{6t+1} \times \{1,2,3,4\}\right ) \cup (\{\infty\} \times \{1,2,3\}),
  \]
  where the second elements in the Cartesian products, $\{1,2,3,4\}$ or $\{1,2,3\}$,  are written as subscripts.
  We let
  \begin{align*}
    P &= \left (\Z_{6t+1} \times \{1\}\right) \cup \{\infty_1\} \\
    A' &= \left (\Z_{6t+1} \times \{2\}\right ) \cup \{\infty_2\} \\
    U' &= \left (\Z_{6t+1} \times \{3,4\}\right)
  \end{align*}
  When $a \equiv 2 \pmod 6$, $A = A'$ and $U = U' \cup \{\infty_3\}$. When $a \equiv 3 \pmod 6$, $A = A' \cup \{\infty_3\}$ and $U = U'$.  When $a \equiv 1 \pmod 6$, $A = A' \setminus \{0_2\}$ and $U = U' \cup \{0_2,\infty_3\}$. We will construct an initial  $\STS(24t+7)$ which is preserved by the action of $\Z_{6t+1}$.  When $a \equiv 3 \pmod 6$ this $\STS$ will contain an embedding of $K_{1,a-1}$. But when $a \equiv 1,2 \pmod 6$, we will have to modify it to ensure that all points in $U$ appear on at least one block of $PPU$. The element $g \in \Z_{6t+1}$ acts on this set via $g(i_{j}) = (i+g)_j$ and $g(\infty_i) = \infty_i$. In each case, we will first present the initial $\STS(24t+7)$ by giving one block per orbit of the action.  We will subsequently describe how the blocks of the $\STS$ must be modified to satisfy our embedding requirements when $a \equiv 1,2 \pmod 6$.
  \begin{itemize}
\item Let  $t \equiv 0,3 \pmod 4$ ($a \equiv 1,2,3,19,20,21 \pmod {24}$).  Let $(a_r,b_r)$ be a Skolem sequence of order $3t$ from \cref{special_skolem} where $b_1 = a_1 + 1 = 2$. Let $\{0,c_r,d_r\}$, $1 \leq r \leq t$ be the base blocks of a cyclic $\STS(6t+1)$ \cite{Handbook}. The left of \cref{three_a_case} shows the base blocks of a $\STS(24t+7)$ which is sufficient for our needs when $a \equiv 3,21 \pmod {24}$.  It is {\it almost} sufficient for our needs when $a \equiv 2,20\pmod {24}$: All edges of $K_{1,a-1}$ are in $PAA$ blocks, all non-edges are in $AAU$ blocks and all edges in $P$ are in blocks of $PPU$.  Every point of the form $i_3,i_4 \in U$ for $i \in \Z_{6t+1}$ is in at least one block from $PPU$ but $\infty_3$ is not in any block of $PPU$.  To establish this last needed property, we remove the blocks
\[
  \{0_1,1_1,(b_1)_4\}, \{0_3,1_3,(b_1)_4\}, \{0_1,0_3,\infty_3\}, \{1_1,1_3,\infty_3\},
\]
and add the blocks
\[
  \{0_1,1_1,\infty_3\}, \{0_3,1_3,\infty_3\}, \{0_1, 0_3, (b_1)_4\},\{1_1,1_3,(b_1)_4\}
\]
which cover exactly the same pairs.  Now $\infty_3 \in U$ is on the $PPU$ block $\{0_1,1_1,\infty_3\}$.

When $a \equiv 1,19 \pmod {24}$, $0_2 \in U$ but does not appear on any $PPU$ block and points $0_2$ and $\infty_2$ are on a $PAA$ thus making an edge in the graph.  We want to remove this edge from the graph and move $0_2$ to $U$.  Here is where we use the fact that $b_1 = a_1 + 1 = 2$.  Remove the blocks
\[
\{\infty_1, 0_2, 0_4\}, \{-2_1, 0_2, -1_3\}, \{\infty_1, -2_1, -2_3\}, \{-1_3,-2_3,0_4\},
\]
and add the blocks
\[
\{\infty_1,-2_1,0_2\}, \{-2_1,-1_3,-2_3\}, \{\infty_1, -2_3,0_4\}, \{0_2,-1_3,0_4\}.
\]
The first of these new blocks is a $PPU$ block forcing $0_2 \in U$ and thus removing it from the graph.

\item Let $t \equiv 1,2 \pmod 4$ ($a \equiv 7,8,9,13,14,15 \pmod {24}$).  Let $(a_r,b_r)$ be a hooked Skolem sequence of order $3t$ from \cref{special_hooked} where $b_2 = a_2 + 2 = 3$. Let $\{0,c_r,d_r\}$, where $1 \leq r \leq t$, be the base blocks of a cyclic $\STS(6t+1)$ \cite{Handbook}. The base blocks of a $\STS(24t+7)$ are presented in the right of \cref{three_a_case} but, as before, this is sufficient for our embedding when $a \equiv 9,15 \pmod {24}$ but needs some adjustment to properly embed $K_{1,a-1}$ when $a \equiv 8,14 \pmod {24}$.  In this case, $\infty_3$ is not in any block of $PPU$.  To ensure this last needed property, we remove the blocks
\[
  \{0_1,1_1,(b_1+1)_4\}, \{0_3,1_3,(b_1+1)_4\}, \{0_1,0_3,\infty_3\}, \{1_1,1_3,\infty_3\},
\]
and add the blocks
\[
  \{0_1,1_1,\infty_3\}, \{0_3,1_3,\infty_3\}, \{0_1, 0_3, (b_1+1)_4\},\{1_1,1_3,(b_1+1)_4\}
\]
which cover exactly the same pairs.  Now $\infty_3 \in U$ is on the $PPU$ block $\{0_1,1_1,\infty_3\}$.

When $a \equiv 7,13 \pmod {24}$, we have $0_2 \in U$ but it does not appear on any $PPU$ block and points $0_2$ and $\infty_2$ are on a $PAA$, thus making an edge in the graph.  We want to remove this edge from the graph and move $0_2$ to $U$.  Here is where we use the fact that $b_2 = a_2 + 2 = 3$.  Remove the blocks
\[
\{\infty_1, 0_2, 0_4\}, \{-4_1, 0_2, -2_3\}, \{\infty_1, -4_1, -4_3\}, \{-2_3,-4_3,0_4\},
\]
and add the blocks
\[
\{\infty_1,-4_1,0_2\}, \{-4_1,-2_3,-4_3\}, \{\infty_1, -4_3,0_4\}, \{0_2,-2_3,0_4\}.
\]
The first of these new blocks is a $PPU$ block forcing $0_2 \in U$ and thus removing it from the graph. \qedhere

\renewcommand{\arraystretch}{1.5}
\begin{table}[ht]
  \begin{center}
\begin{tabular}{cc}
\begin{tabular}{|c|>{\centering\arraybackslash}p{.31\textwidth}|}\hline
  \multirow{2}{*}{$PPU$} & $\{0_1,r_1,(b_r)_4\},\; 1 \leq r \leq 3t $ \\
  & $\{\infty_1, 0_1, 0_3\}$ \\\hline
{$PAA$ }& $ \{\infty_2,0_1,0_2\} $ \\\hline
{$PAU$ or $PAA$} & $\{\infty_1,\infty_2,\infty_3\}$ \\ \hline
  \multirow{2}{*}{$PAU$}& $ \{0_1,(2r)_2,r_3\},\; 1 \leq r \leq 6t, $\\
  & $\{\infty_1,0_2,0_4\}$\\\hline
$PUU$ or $PAU$& $ \{\infty_3,0_1,0_4\} $   \\\hline
  {$AAU$} & $\{0_2,r_2,(b_r)_4\},\; 1 \leq r \leq 3t$ \\ \hline
  {$AUU$ or $AAU$} & $ \{\infty_3,0_2,0_3\} $    \\\hline
  {$AUU$} & $ \{\infty_2,0_3,0_4\} $    \\\hline
  \multirow{2}{*}{$UUU$}& $\{0_3,r_3,(b_r)_4\},\; 1 \leq r \leq 3t$ \\
 & $ \{0_4,(c_r)_4,(d_r)_4\},\; 1 \leq r \leq t$b \\\hline
\end{tabular}
&
\begin{tabular}{|c|>{\centering\arraybackslash}p{.31\textwidth}|}\hline
  \multirow{2}{*}{$PPU$} & $\{0_1,r_1,(b_r+1)_4\},\; 1 \leq r \leq 3t $\\
  & $\{\infty_1, 0_1, 0_3\}$ \\\hline
{$PAA $}& $ \{\infty_2,0_1,0_2\} $ \\\hline
{$PAU$ or  $PAA$} & $\{\infty_1,\infty_2,\infty_3\} $\\ \hline
  \multirow{2}{*}{$PAU$}& $ \{0_1,(2r)_2,r_3\},\; 1 \leq r \leq 6t $\\
  & $ \{\infty_1,0_2,0_4\}$ \\\hline
$PUU$ or $PAU$& $ \{\infty_3,0_1,0_4\} $   \\\hline
  {$AAU$} &$ \{0_2,r_2,(b_r+1)_4\},\; 1 \leq r \leq 3t $\\ \hline
  {$AUU$  or $AAU$} & $ \{\infty_3,0_2,0_3\} $    \\\hline
  {$AUU$} & $ \{\infty_2,0_3,0_4\} $   \\\hline
  \multirow{2}{*}{$UUU$}& $\{0_3,r_3,(b_r+1)_4\},\; 1 \leq r \leq 3t $\\
 &  $\{0_4,(c_r)_4,(d_r)_4\},\; 1 \leq r \leq t$ \\\hline
\end{tabular}
\end{tabular}
\end{center}
\caption{The blocks of $\STS(24t+7)$ when $t \equiv 0,3 \pmod 4$ (left) and $t \equiv 1,2 \pmod{4}$ (right), respectively. \label{three_a_case}}
\end{table}
\renewcommand{\arraystretch}{1}
\end{itemize}

\end{proof}

\begin{lemma}
 Let $5 \leq a \equiv 5 \pmod 6$. Then $K_{1,a-1}$ can be embedded in a $\STS(4a-5)$. This embedding is the minimal possible embedding in a $\STS$.
\end{lemma}
\begin{proof}
When $a \equiv 5 \pmod 6$, the fact that $v \equiv 1,3 \pmod 6$ gives that $v \geq 4a-5$.   We construct the minimal possible embedding of $K_{1,a-1}$.  

When $a \equiv 5 \pmod 6$, let $a = 6t+5$.  We embed the star in a $\STS(24t+15)$ with point set  
  \[
    V = \left (\Z_{6t+3} \times \{1,2,3,4\}\right ) \cup \left( \{\infty\} \times \{1,2,3\}\right),
  \]
  where the second elements in the Cartesian products, $\{1,2,3,4\}$ and $\{1,2,3\}$,  are written as subscripts.
  We let
  \begin{align*}
    P &= \left (\Z_{6t+3} \times \{1\}\right) \cup \{\infty_1\}, \\
    A &= \left (\Z_{6t+3} \times \{2\}\right ) \cup \{\infty_2, \infty_3\}, \\
    U &= \left (\Z_{6t+3} \times \{3,4\}\right).
  \end{align*}

\begin{itemize}
\item When $t \equiv 0,1 \pmod{4}$ ($a \equiv 5, 11 \pmod{24}$),  let $(a_r,b_r)$ be a Skolem sequence of order $3t+1$.  Let the base blocks of a cyclic $\STS(6t+3)$ be $\{c_r,d_r,e_r\}$, where $1 \leq r \leq t+1$ \cite{Handbook}. The left side of \cref{5_6} shows the base blocks of the $\STS(24t+15)$.  These blocks satisfy the necessary conditions to embed the graph $K_{1,a-1}$.

\item When $t \equiv 2,3 \pmod{4}$ ($a \equiv 17, 23 \pmod{24}$),  let $(a_r,b_r)$ be a hooked Skolem sequence of order $3t+1$.  Let the base blocks of a cyclic $\STS(6t+3)$ be $\{c_r,d_r,e_r\}$, where $1 \leq r \leq t+1$ \cite{Handbook}. The right side of \cref{5_6} shows the base blocks of the $\STS(24t+15)$.  These blocks satisfy the necessary conditions to embed the graph $K_{1,a-1}$. \qedhere

\renewcommand{\arraystretch}{1.5}
\begin{table}[ht]
\begin{center}
\begin{tabular}{cc}
\begin{tabular}{|c|>{\centering\arraybackslash}p{0.39\textwidth}|}\hline
  \multirow{2}{*}{$PPU$} & $\{0_1,r_1,(b_r)_4\},\; 1 \leq r \leq 3t+1 $\\
  & $\{\infty_1, 0_1, 0_3\} $\\\hline
{$PAA$}& $ \{\infty_2,0_1,0_2\},  \{\infty_1,\infty_2,\infty_3\} $\\ \hline
  \multirow{2}{*}{$PAU$}& $ \{0_1,(2r)_2,r_3\},\; 1 \leq r \leq 6t+2 $\\
  & $ \{\infty_1,0_2,0_4\} ,  \{\infty_3,0_1,0_4\}  $  \\\hline
 \multirow{2}{*}{$AAU$} & $\{0_2,r_2,(b_r)_4\},\; 1 \leq r \leq 3t+1$ \\
   & $ \{\infty_3,0_2,0_3\}  $   \\\hline
  {$AUU$} &  $\{\infty_2,0_3,0_4\} $   \\\hline
  \multirow{2}{*}{$UUU$}& $\{0_3,r_3,(b_r)_4\},\; 1 \leq r \leq 3t+1 $\\
 & $ \{0_4,(c_r)_4,(d_r)_4\},\; 1 \leq r \leq t+1$ \\\hline
\end{tabular}
&
\begin{tabular}{|c|>{\centering\arraybackslash}p{0.39\textwidth}|}\hline
  \multirow{2}{*}{$PPU$} & $\{0_1,r_1,(b_r+1)_4\},\; 1 \leq r \leq 3t+1 $\\
  & $\{\infty_1, 0_1, 0_3\} $\\\hline
{$PAA$}& $ \{\infty_2,0_1,0_2\},  \{\infty_1,\infty_2,\infty_3\}$ \\ \hline
  \multirow{2}{*}{$PAU$}& $ \{0_1,(2r)_2,r_3\},\; 1 \leq r \leq 6t+2$ \\
  & $ \{\infty_1,0_2,0_4\} ,  \{\infty_3,0_1,0_4\}  $  \\\hline
 \multirow{2}{*}{$AAU$} & $\{0_2,r_2,(b_r+1)_4\},\; 1 \leq r \leq 3t+1 $\\
   &  $\{\infty_3,0_2,0_3\}$     \\\hline
  {$AUU$} & $ \{\infty_2,0_3,0_4\} $   \\\hline
  \multirow{2}{*}{$UUU$}& $\{0_3,r_3,(b_r+1)_4\},\; 1 \leq r \leq 3t+1 $\\
 & $ \{0_4,(c_r)_4,(d_r)_4\},\; 1 \leq r \leq t+1$ \\\hline
\end{tabular}
\end{tabular}
  \end{center}
\caption{The blocks of $\STS(24t+15)$ when $t \equiv 0,1 \pmod{4}$ (left) and $t \equiv 2,3 \pmod {4}$ (right), respectively. \label{5_6}}
\end{table}
\renewcommand{\arraystretch}{1}

\end{itemize}

\end{proof}

These three lemmas establish the existence of very small embeddings for any star.
\begin{theorem} \label{thm_star}
For all $a \leq 9$, the star $K_{1,a-1}$ can be embedded minimally in a $\STS$. For $9 \leq a$, if $a \equiv 0 ,4 \pmod 6$, then $K_{1,a-1}$ can be embedded in a $\STS(4a-3)$. If $a \equiv 1 \pmod 6$, then $K_{1,a-1}$ can be embedded in a $\STS(4a+3)$.  If $a \equiv 2 \pmod 6$, then $K_{1,a-1}$ can be embedded in a $\STS(4a-1)$. If $a \equiv 3,5 \pmod 6$, then $K_{1,a-1}$ can be embedded in a $\STS(4a-5)$. When $a \equiv 3,4,5 \pmod 6$ this embedding is minimal, otherwise this embedding is in a $\STS$ of the next smallest possible order when $a \equiv 0,2 \pmod 6$ and the third smallest possible order when $a \equiv 1 \pmod 6$.
\end{theorem}

One possibility to improve these results to get closer to the minimal embedding is to reduce the order of the cyclic group and correspondingly increase the number of fixed points.  A constraint to doing this is the fact that the number of fixed points must be equivalent to $1,3 \pmod 6$ \cite{colbourn_triple_1999}. For example when $a \equiv 0,2 \pmod 6$ the minimal embedding would be in a $\STS(4a-5)$.  To construct this we could use four copies of $\mathbb{Z}_{a-3}$ with seven fixed points, three of which are in $A$.  Since $7 \equiv 1 \pmod 6$, this is permitted by the constraint.  When $a \equiv 1 \pmod 6$ the minimal embedding would be in a $\STS(4a-3)$. We cannot construct this with $\mathbb{Z}_{a-2}$ and five fixed points because $5 \not\equiv 1,3 \pmod 6$. One option is to decrease the size of the cyclic group further and have either nine or thirteen fixed points, both of which satisfy the constraint.

\section{Experiments and other graphs} \label{sec_other} 

We have conducted a series of experiments in SageMath, the free open-source mathematics software system, to find small embeddings for other families of graphs \cite{sagemath}.  The first step is a full implementation of all the bounds given in \cref{sec_restrictions} combined with the constraints that $p$, $a$, and $u$ must be positive integers (as long as the graph has at least two vertices) and $v \equiv 1,3 \pmod 6$.  For any given graph, we can determine all admissible parameters: $v$, $p$, $a$, $u$, $|PPU|$, $|PAA|$, $|PAU|$, $|PUU|$, $|AAU|$, $|AUU|$, and $|UUU|$ from the smallest possible $v$ up to a chosen largest $v$. 

To find small embeddings for different graphs, we have implemented two search strategies.  The first strategy is to search the entire game trees of \nofil{} played on Steiner triple systems of tractable sizes.  As we traverse the trees, we keep a record of all the available hypergraphs that are graphs.  For $v < 19$ we do this for all $\STS$ and for $19 \leq v$ we sample the space of all $\STS$ by building triple systems using a hillclimb with randomization. For any given graph we can determine the smallest $\STS$ where we have found this graph.  If we find the graph in a $\STS(v)$ for $v \leq 19$, then we know this is the smallest possible $v$. But since we are not fully sampling the spaces of $\STS(v)$ for $19 \leq v$, if the graph we find does not match the smallest admissible $v$ from analysis of the bounds, then we cannot be sure that we have identified the smallest possible embedding.

In the second strategy, for a given graph $G$ and choice of admissible $p$ and $u$, we build a set of triples that is sufficient for an embedding of $G$ if the set of triples can be completed to a full $\STS$.  To build the triples we construct edge colourings of $K_p$, $E(G)$, and $E(\overline{G})$ of sizes $u$, $p$, and $u$, respectively.  These colourings determine the sets of triples $PPU$, $PAA$, and $AAU$, respectively. Since every vertex in $U$ must be on a triple of $PPU$ we must have that every colour class in $K_p$ is non-empty.  We then use a hillclimb to try to complete the given set of triples to a $\STS$.  Based on the intuition that a more evenly distributed set of triples is easier to complete, we construct the initial set of triples from equitable edge colourings of the needed sizes.  This search proceeds from the smallest admissible $v$ up to a user specified maximum $v$ and for a fixed $v$ can prioritize a large $p$ or prioritize a large $u$ in its loop over possible parameters sets.  Since we are not exhaustively considering all possible relevant edge colourings to build the initial triples and as there is randomization in attempting to complete these initial triples to a Steiner triple system, we cannot claim that a given embedding is the smallest possible unless $v$ is the smallest admissible size.  The source code for all these implementations is available from the authors on request.

\subsection{Embedding empty graphs, $\overline{K_a}$}\label{empty_experiments}

\cref{empty_params} shows the parameters of possible target minimal embeddings for empty graphs.  The minimum possible $v$ does not grow linearly but rather as $2a + 2 + \sqrt{8a+1}$.  This eliminates the possibility of using an easy cyclic construction to construct minimal embeddings.  A  ``small'' cyclic construction might have $P$, $A$, and $U$ be one, two, and three copies, respectively, of $\mathbb{Z}_{n}$ for $n \approx a/2$ with fixed points as appropriate.  Recursive constructions, rather than direct constructions, might permit getting closer to the minimum values of $v$.  

The minimal embeddings matching parameters from \cref{empty_params} for $a=2,3,6$ do not exist, but these graphs are embedded in a $\STS(15)$, $\STS(15)$, and $\STS(25)$, respectively, which are each the next possible Steiner triple system order. Minimal embeddings matching \cref{empty_params} exist for $a=4,5$.  For $a=2,3,6$ the non-existence of an embedding matching parameters from \cref{empty_params} are for the same reason: in each case the parameters force every pair in $U$ (which would have sizes 6, 6, and 10 respectively)  to be in a $UUU$ triple which would imply the existence of Steiner triple systems of orders 6 and 10, which cannot exist. These are the first cases we have encountered where a structural constraint prohibits the existence of a $\STS$ meeting the bounds from \cref{u_bounds}.  In fact, looking at \cref{empty_params}, most of the rows have $PUU$ and $AUU$ empty, similarly forcing the $UUU$ blocks to form a $\STS$ and implying that $u \equiv 1,3 \pmod 6$ is a necessary condition.  This shows that twenty-one rows are parameters for which an embedding cannot exist.  For those parameter sets where the $UUU$ blocks must be a $\STS$, this structure could aid in finding embeddings. This includes $a=7$ where $v=25$ is impossible and our searches verify that an embedding in the next smallest $\STS(27)$ does exist.  The parameters for $a=8,9$ both require the $UUU$ blocks to be a $\STS$, and the orders $13$ and $15$ are admissible. In both these cases minimal embeddings exist.  Neither of our searches could find the minimal embedding when $a=9$. We constructed it by hand using a cyclic system with four copies of $\mathbb{Z}_7$ and three fixed points.  Taking advantage of the cyclicity and the fact that any $\STS(15)$ could be taken for the $UUU$ blocks made this feasible.  All minimal embeddings for $2 \leq a < 10$ can be found in the embedding appendix file. 

\begin{table}
\[
   \begin{array}{|c|c|c|c|}
\hline
a & v & (p,a,u) & (PPU,PAA,PAU,PUU,AAU,AUU,UUU) \\
\hline
2 & 13 & (5, 2, 6) & (10, 0, 10, 0, 1, 0, 5) \\
3 & 13 & (4, 3, 6) & (6, 0, 12, 0, 3, 0, 5) \\
4 & 19 & (6, 4, 9) & (15, 0, 24, 0, 6, 0, 12) \\
&&(5, 4, 10) & (10, 0, 20, 5, 6, 4, 12) \\
5 & 19 & (5, 5, 9) & (10, 0, 25, 0, 10, 0, 12) \\
6 & 21 & (5, 6, 10) & (10, 0, 30, 0, 15, 0, 15) \\
7 & 25 & (6, 7, 12) & (15, 0, 42, 0, 21, 0, 22) \\
8 & 27 & (6, 8, 13) & (15, 0, 48, 0, 28, 0, 26) \\
9 & 31 & (7, 9, 15) & (21, 0, 63, 0, 36, 0, 35) \\
10 & 31 & (6, 10, 15) & (15, 0, 60, 0, 45, 0, 35) \\
11 & 37 & (8, 11, 18) & (28, 0, 88, 0, 55, 0, 51) \\
&&(7, 11, 19) & (21, 0, 77, 7, 55, 11, 51)\\
12 & 37 & (7, 12, 18) & (21, 0, 84, 0, 66, 0, 51) \\
13 & 39 & (7, 13, 19) & (21, 0, 91, 0, 78, 0, 57) \\
14 & 43 & (8, 14, 21) & (28, 0, 112, 0, 91, 0, 70) \\
15 & 43 & (7, 15, 21) & (21, 0, 105, 0, 105, 0, 70) \\
16 & 49 & (9, 16, 24) & (36, 0, 144, 0, 120, 0, 92) \\
&& (8, 16, 25) &(28, 0, 128, 8, 120, 16, 92)\\
17 & 49 & (8, 17, 24) & (28, 0, 136, 0, 136, 0, 92) \\
18 & 51 & (8, 18, 25) & (28, 0, 144, 0, 153, 0, 100) \\
19 & 55 & (9, 19, 27) & (36, 0, 171, 0, 171, 0, 117) \\
&&(8, 19, 28)& (28, 0, 152, 8, 171, 19, 117)\\
20 & 55 & (8, 20, 27) & (28, 0, 160, 0, 190, 0, 117) \\
21 & 57 & (8, 21, 28) & (28, 0, 168, 0, 210, 0, 126) \\
22 & 61 & (9, 22, 30) & (36, 0, 198, 0, 231, 0, 145) \\
23 & 63 & (9, 23, 31) & (36, 0, 207, 0, 253, 0, 155) \\
24 & 67 & (10, 24, 33) & (45, 0, 240, 0, 276, 0, 176) \\
&&(9, 24, 34)& (36, 0, 216, 9, 276, 24, 176)\\
25 & 67 & (9, 25, 33) & (36, 0, 225, 0, 300, 0, 176) \\
26 & 69 & (9, 26, 34) & (36, 0, 234, 0, 325, 0, 187) \\
27 & 73 & (10, 27, 36) & (45, 0, 270, 0, 351, 0, 210) \\
28 & 73 & (9, 28, 36) & (36, 0, 252, 0, 378, 0, 210) \\
29 & 79 & (11, 29, 39) & (55, 0, 319, 0, 406, 0, 247) \\
&& (10, 29, 40) & (45, 0, 290, 10, 406, 29, 247)\\
30 & 79 & (10, 30, 39) & (45, 0, 300, 0, 435, 0, 247) \\
31 & 81 & (10, 31, 40) & (45, 0, 310, 0, 465, 0, 260) \\
32 & 85 & (11, 32, 42) & (55, 0, 352, 0, 496, 0, 287) \\
&&(10, 32, 43) & (45, 0, 320, 10, 496, 32, 287)\\
33 & 85 & (10, 33, 42) & (45, 0, 330, 0, 528, 0, 287) \\
34 & 87 & (10, 34, 43) & (45, 0, 340, 0, 561, 0, 301) \\
35 & 91 & (11, 35, 45) & (55, 0, 385, 0, 595, 0, 330) \\
36 & 91 & (10, 36, 45) & (45, 0, 360, 0, 630, 0, 330) \\
37 & 97 & (12, 37, 48) & (66, 0, 444, 0, 666, 0, 376) \\
&&  (11, 37, 49) & (55, 0, 407, 11, 666, 37, 376)\\
38 & 97 & (11, 38, 48) & (55, 0, 418, 0, 703, 0, 376) \\
39 & 99 & (11, 39, 49) & (55, 0, 429, 0, 741, 0, 392) \\
40 & 103 & (12, 40, 51) & (66, 0, 480, 0, 780, 0, 425) \\
&& (11, 40, 52) & (55, 0, 440, 11, 780, 40, 425)\\
41 & 103 & (11, 41, 51) & (55, 0, 451, 0, 820, 0, 425) \\
42 & 105 & (11, 42, 52) & (55, 0, 462, 0, 861, 0, 442) \\
43 & 109 & (12, 43, 54) & (66, 0, 516, 0, 903, 0, 477) \\
&& (11, 43, 55)  & (55, 0, 473, 11, 903, 43, 477)\\
44 & 109 & (11, 44, 54) & (55, 0, 484, 0, 946, 0, 477) \\
45 & 111 & (11, 45, 55) & (55, 0, 495, 0, 990, 0, 495) \\
%46 & 115 & (12, 46, 57) & (66, 0, 552, 0, 1035, 0, 532) \\
%47 & 117 & (12, 47, 58) & (66, 0, 564, 0, 1081, 0, 551) \\
%48 & 121 & (13, 48, 60) & (78, 0, 624, 0, 1128, 0, 590) \\
%&&  (12, 48, 61) & (66, 0, 576, 12, 1128, 48, 590)\\
%49 & 121 & (12, 49, 60) & (66, 0, 588, 0, 1176, 0, 590) \\
%50 & 123 & (12, 50, 61) & (66, 0, 600, 0, 1225, 0, 610) \\
\hline
\end{array} 
\]
\caption{Smallest admissible parameters for embedding empty graphs, $\overline{K_{a}}$. \label{empty_params}}
\end{table}

\subsection{Embedding paths, $P_a$}\label{path_experiments}

\cref{path_params} shows the parameters of possible target minimal embeddings for path graphs.  The minimum possible $v$ does not grow linearly but rather as $2a + 2 - 2(a-1)/a + \sqrt{8a+1 - 32(a-1)/a}$ \cite{HugganHS2021}.  This, and the fact that $P_a$ does not have any cyclic automorphisms of order larger than 2, eliminates the possibility of using an easy cyclic construction to construct minimal embeddings.  Recursive constructions might have utility.  When $a=2,3$, then $P_2 = K_2$ and $P_3 = K_{1,2}$, so these minimal embeddings have been discussed earlier.  Embeddings matching the minimum parameters given in \cref{path_params} exist for $a=4,5,6,7,8,9$.  It is interesting to note that the minimal embedding for $a=4$ is larger than the one possible for $a=5$. So for this family of graphs the lower bound on $v$ (which is tight for these small values of $a$) is not monotonic in $a$. The easiest way to understand this is to note that \cref{puu_1} and \cref{puu_2} derive from the fact that $|PUU| \geq 0$ \cite{HugganHS2021} and an equivalent formulation of this is that
\[
u \geq p + a -1 - \frac{2e}{p}.
\]
We also note that \cref{u_binom_p_2} is equivalent to 
\[
u \leq \binom{p}{2},
\]
thus
\begin{equation} \label{u_interval}
p + a-1 - \frac{2e}{p} \leq u \leq \binom{p}{2}.
\end{equation}

These lower and upper bounds on $u$ imply that the interval must exist and thus
\[
0 \leq p^3 - 3p^2 - 2(a-1)p + 4e,
\]
which in the case of path graphs is
\begin{equation} \label{p_bound}
    0 \leq p^3 - 3p^2 - 2(a-1)p + 4(a-1).
\end{equation}
The non-monotonicity comes from the interplay of this lower bound on $p$ and the fact that $v$ must be a positive integer such that $v\equiv 1,3 \pmod 6$.  In path graphs when $a$ is either 4 or 5, \cref{p_bound} forces $4 \leq p$, and when $p=4$, \cref{u_interval} forces $u = 6$.  When $a=4$, we get that $v = p + a + u = 14 \not \equiv 1,3 \pmod 6$ and so $5 \leq p$.  However, when $p=5$ the lower bound from \cref{u_interval} gives $7 \leq u$ and thus $v \geq 16$ implying $v \geq 19$, where an embedding meeting the lower bound exists.  In contrast when $a=5$, we get  $v = p + a + u = 15 \equiv 1,3 \pmod 6$ and so $v \geq 15$ where a matching embedding also exists.  All minimal embeddings for $4 \leq a < 10$ can be found in the embedding appendix file.

\begin{table}
\[
\begin{array}{|c|c|c|c|}
\hline
a & v & (p,a,u) & (PPU,PAA,PAU,PUU,AAU,AUU,UUU) \\
\hline
2 & 13 & (5, 2, 6) & (10, 1, 8, 1, 0, 2, 4) \\
3 & 13 & (4, 3, 6) & (6, 2, 8, 2, 1, 4, 3) \\
4 & 19 & (6, 4, 9) & (15, 3, 18, 3, 3, 6, 9) \\
 &  & (5, 4, 10) & (10, 3, 14, 8, 3, 10, 9) \\
5 & 15 & (4, 5, 6) & (6, 4, 12, 0, 6, 3, 4) \\
6 & 19 & (5, 6, 8) & (10, 5, 20, 0, 10, 4, 8) \\
7 & 21 & (5, 7, 9) & (10, 6, 23, 1, 15, 5, 10) \\
8 & 25 & (6, 8, 11) & (15, 7, 34, 1, 21, 6, 16) \\
9 & 27 & (6, 9, 12) & (15, 8, 38, 2, 28, 7, 19) \\
10 & 31 & (7, 10, 14) & (21, 9, 52, 2, 36, 8, 27) \\
 &  & (6, 10, 15) & (15, 9, 42, 9, 36, 18, 26) \\
11 & 31 & (6, 11, 14) & (15, 10, 46, 4, 45, 9, 26) \\
12 & 33 & (6, 12, 15) & (15, 11, 50, 5, 55, 10, 30) \\
13 & 37 & (7, 13, 17) & (21, 12, 67, 5, 66, 11, 40) \\
14 & 39 & (7, 14, 18) & (21, 13, 72, 6, 78, 12, 45) \\
15 & 43 & (8, 15, 20) & (28, 14, 92, 6, 91, 13, 57) \\
 &  & (7, 15, 21) & (21, 14, 77, 14, 91, 28, 56) \\
16 & 43 & (7, 16, 20) & (21, 15, 82, 8, 105, 14, 56) \\
17 & 45 & (7, 17, 21) & (21, 16, 87, 9, 120, 15, 62) \\
18 & 49 & (8, 18, 23) & (28, 17, 110, 9, 136, 16, 76) \\
19 & 51 & (8, 19, 24) & (28, 18, 116, 10, 153, 17, 83) \\
20 & 55 & (9, 20, 26) & (36, 19, 142, 10, 171, 18, 99) \\
 &  & (8, 20, 27) & (28, 19, 122, 19, 171, 38, 98) \\
21 & 55 & (8, 21, 26) & (28, 20, 128, 12, 190, 19, 98) \\
22 & 57 & (8, 22, 27) & (28, 21, 134, 13, 210, 20, 106) \\
23 & 61 & (9, 23, 29) & (36, 22, 163, 13, 231, 21, 124) \\
24 & 63 & (9, 24, 30) & (36, 23, 170, 14, 253, 22, 133) \\
25  & 67  & (10, 25, 32)  & (45, 24, 202, 14, 276, 23, 153) \\
  &&  (9, 25, 33) & (36, 24, 177, 24, 276, 48, 152) \\
 26 & 67 & (9, 26, 32) & (36, 25, 184, 16, 300, 24, 152) \\
 27 & 69 & (9, 27, 33) & (36, 26, 191, 17, 325, 25, 162) \\
 28 & 73 & (10, 28, 35) & (45, 27, 226, 17, 351, 26, 184) \\
   && (9, 28, 36) & (36, 27, 198, 27, 351, 54, 183) \\
 29 & 73 & (9, 29, 35) & (36, 28, 205, 19, 378, 27, 183) \\
 30 & 75 & (9, 30, 36) & (36, 29, 212, 20, 406, 28, 194) \\
 31 & 79 & (10, 31, 38) & (45, 30, 250, 20, 435, 29, 218) \\
 32 & 81 & (10, 32, 39) & (45, 31, 258, 21, 465, 30, 230) \\
 33 & 85 & (11, 33, 41) & (55, 32, 299, 21, 496, 31, 256) \\
   && (10, 33, 42) & (45, 32, 266, 32, 496, 64, 255) \\
 34 & 85 & (10, 34, 41) & (45, 33, 274, 23, 528, 32, 255) \\
 35 & 87 & (10, 35, 42) & (45, 34, 282, 24, 561, 33, 268) \\
 36 & 91 & (11, 36, 44) & (55, 35, 326, 24, 595, 34, 296) \\
   && (10, 36, 45) & (45, 35, 290, 35, 595, 70, 295) \\
 37 & 91 & (10, 37, 44) & (45, 36, 298, 26, 630, 35, 295) \\
 38 & 93 & (10, 38, 45) & (45, 37, 306, 27, 666, 36, 309) \\
 39 & 97 & (11, 39, 47) & (55, 38, 353, 27, 703, 37, 339) \\
 40 & 99 & (11, 40, 48) & (55, 39, 362, 28, 741, 38, 354) \\
 41 & 103 & (12, 41, 50) & (66, 40, 412, 28, 780, 39, 386) \\
   && (11, 41, 51) & (55, 40, 371, 40, 780, 80, 385) \\
 42 & 103 & (11, 42, 50) & (55, 41, 380, 30, 820, 40, 385) \\
 43 & 105 & (11, 43, 51) & (55, 42, 389, 31, 861, 41, 401) \\
 44 & 109 & (12, 44, 53) & (66, 43, 442, 31, 903, 42, 435) \\
   && (11, 44, 54) & (55, 43, 398, 43, 903, 86, 434) \\
 45 & 109 & (11, 45, 53) & (55, 44, 407, 33, 946, 43, 434) \\
 % 46 & 111 & (11, 46, 54) & (55, 45, 416, 34, 990, 44, 451) \\
 % 47 & 115 & (12, 47, 56) & (66, 46, 472, 34, 1035, 45, 487) \\
 % 48 & 117 & (12, 48, 57) & (66, 47, 482, 35, 1081, 46, 505) \\
 % 49 & 121 & (13, 49, 59) & (78, 48, 541, 35, 1128, 47, 543) \\ 
 %   && (12, 49, 60) & (66, 48, 492, 48, 1128, 96, 542) \\
 % 50 & 121 & (12, 50, 59) & (66, 49, 502, 37, 1176, 48, 542) 
\hline
\end{array}
\]
\caption{Smallest admissible parameters for embedding path graphs, $P_{a}$. \label{path_params}}
\end{table}

\subsection{Embedding cycles, $C_a$}\label{cycle_experiments}
\cref{cycle_params} shows the parameters of possible target minimal embeddings for cycle graphs.  The minimum possible $v$ does not grow linearly but rather as $2a -2 + \sqrt{8a-31}$ \cite{HugganHS2021}.  This eliminates the possibility of using an easy cyclic construction to construct minimal embeddings.   A  ``small'' cyclic construction that grows linearly in $a$ might have $P$, $A$, and $U$ be one, two, and three copies, respectively, of $\mathbb{Z}_{n}$ for $n \approx a/2$ with fixed points as appropriate.  When $a=3$ the cycle graph is complete and a minimal embedding is constructed in \cref{k_a_theorem}.  Minimal embeddings matching \cref{cycle_params} exist for $a=4,5,6,7,8,9$. The minimum admissible values of $v$ for $a=9,10$ are $27$ and $25$ respectively; another example of non-monotonicity which happens for the exact same reasons described in \cref{path_experiments}. All minimal embeddings for $4 \leq a < 10$ can be found in the embedding appendix file. 

\begin{table}
\[
\begin{array}{|c|c|c|c|}
\hline
a & v & (p,a,u) & (PPU,PAA,PAU,PUU,AAU,AUU,UUU) \\
\hline
4 & 7 & (2, 4, 1) & (1, 4, 0, 0, 2, 0, 0) \\
5 & 15 & (4, 5, 6) & (6, 5, 10, 1, 5, 5, 3) \\
6 & 19 & (5, 6, 8) & (10, 6, 18, 1, 9, 6, 7) \\
7 & 21 & (5, 7, 9) & (10, 7, 21, 2, 14, 7, 9) \\
8 & 25 & (6, 8, 11) & (15, 8, 32, 2, 20, 8, 15) \\
9 & 27 & (6, 9, 12) & (15, 9, 36, 3, 27, 9, 18) \\
10 & 25 & (5, 10, 10) & (10, 10, 30, 0, 35, 0, 15) \\
11 & 31 & (6, 11, 14) & (15, 11, 44, 5, 44, 11, 25) \\
12 & 31 & (6, 12, 13) & (15, 12, 48, 0, 54, 0, 26) \\
13 & 33 & (6, 13, 14) & (15, 13, 52, 1, 65, 0, 30) \\
14 & 37 & (7, 14, 16) & (21, 14, 70, 0, 77, 0, 40) \\
15 & 39 & (7, 15, 17) & (21, 15, 75, 1, 90, 0, 45) \\
16 & 43 & (8, 16, 19) & (28, 16, 96, 0, 104, 0, 57) \\
 &  & (7, 16, 20) & (21, 16, 80, 9, 104, 16, 55) \\
17 & 43 & (7, 17, 19) & (21, 17, 85, 3, 119, 0, 56) \\
18 & 45 & (7, 18, 20) & (21, 18, 90, 4, 135, 0, 62) \\
19 & 49 & (8, 19, 22) & (28, 19, 114, 3, 152, 0, 76) \\
20 & 51 & (8, 20, 23) & (28, 20, 120, 4, 170, 0, 83) \\
21 & 55 & (9, 21, 25) & (36, 21, 147, 3, 189, 0, 99) \\
 &  & (8, 21, 26) & (28, 21, 126, 13, 189, 21, 97) \\
22 & 55 & (8, 22, 25) & (28, 22, 132, 6, 209, 0, 98) \\
23 & 57 & (8, 23, 26) & (28, 23, 138, 7, 230, 0, 106) \\
24 & 61 & (9, 24, 28) & (36, 24, 168, 6, 252, 0, 124) \\
25 & 61 & (8, 25, 28) & (28, 25, 150, 9, 275, 0, 123) \\
26 & 67 & (10, 26, 31) & (45, 26, 208, 6, 299, 0, 153) \\
 &  & (9, 26, 32) & (36, 26, 182, 17, 299, 26, 151) \\
27 & 67 & (9, 27, 31) & (36, 27, 189, 9, 324, 0, 152) \\
28 & 69 & (9, 28, 32) & (36, 28, 196, 10, 350, 0, 162) \\
29 & 73 & (10, 29, 34) & (45, 29, 232, 9, 377, 0, 184) \\
 &  & (9, 29, 35) & (36, 29, 203, 20, 377, 29, 182) \\
30 & 73 & (9, 30, 34) & (36, 30, 210, 12, 405, 0, 183) \\
31 & 75 & (9, 31, 35) & (36, 31, 217, 13, 434, 0, 194) \\
32 & 79 & (10, 32, 37) & (45, 32, 256, 12, 464, 0, 218) \\
33 & 81 & (10, 33, 38) & (45, 33, 264, 13, 495, 0, 230) \\
34 & 85 & (11, 34, 40) & (55, 34, 306, 12, 527, 0, 256) \\
 &  & (10, 34, 41) & (45, 34, 272, 24, 527, 34, 254) \\
35 & 85 & (10, 35, 40) & (45, 35, 280, 15, 560, 0, 255) \\
36 & 87 & (10, 36, 41) & (45, 36, 288, 16, 594, 0, 268) \\
37 & 91 & (11, 37, 43) & (55, 37, 333, 15, 629, 0, 296) \\
 &  & (10, 37, 44) & (45, 37, 296, 27, 629, 37, 294) \\
38 & 91 & (10, 38, 43) & (45, 38, 304, 18, 665, 0, 295) \\
39 & 93 & (10, 39, 44) & (45, 39, 312, 19, 702, 0, 309) \\ 
40 & 97 & (11, 40, 46) & (55, 40, 360, 18, 740, 0, 339) \\
 41 & 99 & (11, 41, 47) & (55, 41, 369, 19, 779, 0, 354) \\
 42 & 103  & (12, 42, 49)  & (66, 42, 420, 18, 819, 0, 386) \\
    &&     (11, 42, 50)  & (55, 42, 378, 31, 819, 42, 384) \\
 43 & 103  & (11, 43, 49)  & (55, 43, 387, 21, 860, 0, 385) \\
 44 & 105  & (11, 44, 50)  & (55, 44, 396, 22, 902, 0, 401) \\
 45 & 109  & (12, 45, 52)  & (66, 45, 450, 21, 945, 0, 435) \\
      &   & (11, 45, 53)  & (55, 45, 405, 34, 945, 45, 433) \\
 % 46 & 109  & (11, 46, 52)  & (55, 46, 414, 24, 989, 0, 434) \\
 % 47 & 111  & (11, 47, 53)  & (55, 47, 423, 25, 1034, 0, 451) \\
 % 48 & 115  & (12, 48, 55)  & (66, 48, 480, 24, 1080, 0, 487) \\
 % 49 & 115  & (11, 49, 55)  & (55, 49, 441, 27, 1127, 0, 486) \\
 % 50 & 121  & (13, 50, 58)  & (78, 50, 550, 24, 1175, 0, 543) \\
 %     &&     (12, 50, 59) & (66, 50, 500, 38, 1175, 50, 541) \\

\hline
\end{array}
\]
    \caption{Smallest admissible parameters for embedding cycle graphs, $C_{a}$. \label{cycle_params}}
\end{table}

\section{Conclusion}\label{sec_conclusion}

In this article we have shown that for the family of complete graphs, $K_a$, there exist embeddings in Steiner triple systems whose orders exactly meet the bounds derived from \cref{u_bounds}.  For star graphs, $K_{1,a-1}$, we have constructed embeddings that meet the bounds when $a \equiv 3,4,5 \pmod 6$, attain the next possible order of a $\STS$ when $a \equiv 0,2 \pmod 6$, and achieve the third smallest possible order when $a \equiv 1 \pmod 6$.  We have discussed a possible strategy to improve these embeddings to minimal ones using a cyclic construction with four copies of a cyclic group and more than three fixed points.  For both of these graph families the minimal $\STS$ order grows linearly with $a$.  

In \cref{sec_other} we describe our exploration of three other families of graphs where the minimum possible order of a $\STS$ grows as something close to $2a + \sqrt{8a}$: empty graphs, paths, and cycles.  We have designed two different randomized searches, implemented in SageMath, to search for embeddings of graphs in $\STS$.  One proceeds by generating ``random'' $\STS(v)$ via a hill climb and parsing its game tree for what graphs are embedded.  The second builds a collection of blocks which is sufficient to certify an embedding of a given graph and then uses a hill climb to complete these blocks to a $\STS(v)$.  The former seems to work best for smaller graphs and the latter works better for larger instances.  Using these searches and one direct construction, we have determined the minimal embeddings for $\overline{K_a}$, $P_a$, and $C_a$ for all $a \leq 10$.  In empty graphs we have identified a block design structural constraint that is an obstruction to meeting the lower bound from \cref{u_bounds} and a possible aid to finding minimal embeddings when it does not forbid them.  For the paths and cycles all the minimal embeddings meet the minimum values from \cref{u_bounds}.  In both the cases of paths and cycles we observe that the sequence of minimal $\STS$ orders is non-monotonic and explore the reasons which arise from when the bounds do not give $v\equiv 1,3 \pmod 6$.  Due to the squareroot in the probable growth rate for the minimal order, we suggest that recursive constructions might be a better strategy to attempt to construct minimal embeddings for these families of graphs.

Another possible technique that shows promise is using Pasch switches.  A Pasch configuration in a Steiner triple system is a set of four blocks on six points, each point on two of the blocks.  It is unique up to isomorphism with blocks
\[
\{u,v,z\},\quad \{u,y,w\},\quad \{x,v,w\},\quad \{x,y,z\}.
\]
For any Pasch there is a alternate set of four blocks on the same six points that covers the exact same pairs:
\[
\{u,v,w\},\quad \{u,y,z\},\quad \{x,v,z\},\quad \{x,y,w\}.
\]
Replacing the original four blocks with these alternate four yields another Steiner triple system.  In our construction of stars $K_{1,a-1}$, for $a \equiv 1 \pmod 6$ we started with an embedding for $a \equiv 2 \pmod 6$.  Then we found a Pasch where $v, w \in P$, $u \in A$, and $x,y,z \in U$.  After the switch the point $u$ is now on a block with two points from $P$, thus we have transferred the point from $A$ to $U$, therefore deleting it from the graph along with its edges.

The graph families we considered herein were chosen for reaching minimal and maximal edge densities, chromatic index and chromatic index of their complements.  It would be interesting to explore families with more moderate values for these parameters, such as complete bipartite graphs and higher density Cayley graphs.  Both those families of graphs have large automorphism groups which could mean that direct ``cyclic'' methods work well to construct embeddings.  In \cite{HugganHS2021} we used the bounds from \cref{u_bounds} to determine that the minimum admissible $v$ is likely to be a linear function of $a$ when the number of edges is in the interval $[(\sqrt{8a^3-7a^2} - a)/4, a^2/4]$ or grows quadratically in $a$. Exploring families of graphs close to these boundaries to see if the minium embeddings are indeed linear would be interesting.   Graph families that have been studied for \nodekayles{} would also be of some interest \cite{MR4105864}.  

We observed non-monotonicity of the minimum $v$ as a function of $a$ for paths and cycles.  What other families of graphs exhibit this non-monotonicity? The non-monotonicity for both paths and cycles occurred for small values of $a$ and $v$. In other graph families, can these instances of non-monotonicity only occur for relatively small values of $a$?  In our explorations of empty graphs we encountered block design obstructions to meeting the minimum admissible $v$ derived from \cref{u_bounds}.  These came from the fact that if $PUU$ and $AUU$ are empty, then the blocks from $UUU$ must form a $\STS(u)$.  What other design theoretic obstructions might occur for embeddings graphs in Steiner triple systems? In what other families of graphs will we find structural considerations forcing the minimum $v$ to be larger than determined by \cref{u_bounds}?  

The following is an open problem: the embeddings of the graphs we have given here only guarantee that there is some sequence of gameplay from the $\STS$ to the desired graph, without any consideration of optimal play. Thus, the question is whether under optimal play the embedded graphs are indeed appearing in a sequence of optimal play starting from the $\STS$. If not, is there another $\STS$ that guarantees the graph appearing when playing optimally?

Finally, it is unknown whether the so-called \textsc{Nim}-dimension of \nofil{} played on designs is unbounded. This is equivalent to asking whether all possible nim-values occur as the nim value of a block design, or possibly even of a $\STS$. For \nodekayles{} it is known that the \textsc{Nim}-dimension is infinite -- for example \cite{MR4105864} have shown that $2\times n$ Cartesian grids with an added leaf at one end take on all values. As mentioned in the previous paragraph though, optimal game play, as is required for the determining nim-values, might be avoiding those graphs. Thus the problem of possible nim-values of \nofil{} is still open.

\end{document}